\documentclass[3p,12pt,review]{elsarticle}
            \usepackage{color}
\usepackage{hyperref}
\makeatletter

\newcommand{\Rmnum}[1]{\expandafter\@slowromancap\romannumeral #1@}
\makeatother

\usepackage{amsmath,amsthm,amsfonts,amssymb,mathrsfs}
\usepackage{cases}
\biboptions{numbers,sort&compress}
\usepackage{hyperref}
\hypersetup{hidelinks}
\journal{}
\newtheorem{remark}{Remark}[section]
\newtheorem{prop}{Proposition}[section]
\newtheorem{theorem}{Theorem}[section]
\newtheorem{definition}{Definition}[section]
\newtheorem{lemma}{Lemma}[section]
\newtheorem{hypo}{Hypothesis}[section]

\DeclareMathOperator{\esssup}{esssup}
\DeclareMathOperator{\essinf}{essinf}
\DeclareMathOperator{\Tr}{Tr}
\usepackage{comment}
\allowdisplaybreaks[4]

\makeatletter
\renewcommand\subsection{\@startsection {subsection}{2}{\z@}%
                {3.5ex \@plus -1ex \@minus -.2ex}%
                {1.5ex \@plus.2ex\noindent}%
                {\bfseries\itshape}%
                }
\makeatother
\begin{document}

\let\normalint\int 
\def\int{\displaystyle\normalint}
\begin{frontmatter}

\title{{\bf Recursive stochastic differential games with non-Lipschitzian generators and viscosity solutions of Hamilton-Jacobi-Bellman-Isaacs equations}\tnoteref{mytitlenote}}
\tnotetext[mytitlenote]{2020  {\it Mathematics Subject Classification.}
Primary: 49N70; Secondary: 49L20.\\
\hspace*{1.8em}{\it Key words and phrases.} Backward stochastic differential equation; dynamic programming principle; Hamilton-Jacobi-Bellman-Isaacs equation; stochastic differential game; viscosity solution.\\
\hspace*{1.8em}The first edition of this paper is completed in  October, 2021. Then it has been revised in October, 2023. }

\author{Guangchen Wang}

\author{Zhuangzhuang Xing}
\address{School of Control Science and Engineering, Shandong University, Jinan 250061, PR China}



\begin{abstract}
This investigation is dedicated to  a  two-player zero-sum  stochastic differential game (SDG), where a cost function is characterized by a backward stochastic differential equation (BSDE) with a continuous and monotonic generator regarding the first unknown variable, which possesses  immense applicability in financial engineering.  A verification theorem by virtue of classical solution of derived Hamilton-Jacobi-Bellman-Isaacs (HJBI) equation  is given.
The dynamic programming principle (DPP) and  unique weak (viscosity) solvability of HJBI equation   are formulated through  comparison theorem for BSDEs with monotonic generators  and  stability of viscosity solution. Some new regularity properties of value function are presented.
Finally, we propose three concrete examples, which are concerned with resp., classical, and viscosity solution of HJBI equation, as  well as  a financial application  where an investor with a non-Lipschitzian Epstein-Zin utility  deals with market friction to maximize her utility preference.
\end{abstract}
\end{frontmatter}

\section{Introduction}\label{sec1}
The initiation of cogitating on two-player zero-sum differential games  in  perspective of  DPP framework was  by  Evans and Souganidis \cite{Evans1984}. Fleming and  Souganidis \cite{Fleming1989} pioneered to  expand the former results of \cite{Evans1984} from deterministic to stochastic cases. Since there have been copious results following  their approaches, referred to \cite{Tang1993,Kushner2002,Zhang2011,Lv2020}, as well as Fleming and Soner \cite{Fleming2006} for the applications regarding robust, $H_\infty$ and risk sensitive controls. Buckdahn and Li \cite{Buckdahn2008} extended \cite{Fleming1989} to the BSDE framework. The literature interested in  their techniques can be referred to \cite{Buckdahn2011,Lin2012,Wei2018,Lin2013,Ji2013,Li2021} and their references therein.

Meanwhile, the studies of BSDEs have received sharply raising  attention in light of their pivotal roles in financial engineering, risk management and control theories, see Duffie and Epstein \cite{Duffie1992},  El Karoui et al. \cite{Karoui1997},  Wu and Yu \cite{Wu2008}, etc.  Pardoux and Peng \cite{Pardoux1990} formulated  the unique solvability of nonlinear BSDEs in Lipschitzian condition. Nonetheless the Lipschitzian condition is too strict to satisfy in reality.  For instance, the recursive utility introduced in \cite{Duffie1992} defined by a BSDE with  conditional expectation:
\begin{equation}\label{digui}
u_r=E[\int_r^Tg(c_t,u_t)dt|\mathcal{F}_r],
\end{equation}
   in which $c$ represents consumption of an investor. When considering the celebrated Epstein-Zin preference, i.e., $g=\frac{\rho}{1-\frac{1}{\varsigma}}(1-\vartheta)u[(\frac{c}{((1-\vartheta)u)^{\frac{1}{1-\vartheta}}})^{1-\frac{1}{\varsigma}}-1]$, in which $\rho>0$ represents the rate of time preference, $0<\vartheta\neq 1$  the ratio of relative risk aversion and $0<\varsigma\neq 1$ the elasticity of intertemporal substitution, it is likely that $g$ is not Lipschitzian regarding $c$ and $u$ but monotonic regarding $u$ in some circumstances, see Kraft et al. \cite{Kraft2013}.  Additionally, the investor has to overcome  market friction, such as transaction costs, competition, etc., which can be seen as an opposer, to gain profits. In this circumstance the DPPs of recursive control problems in \cite{Wu2008} and of recursive SDGs in \cite{Buckdahn2008} fail to apply, which motivates us to investigate  recursive SDGs with non-Lipschitzian generators.  Some preliminary works concerning relaxing the Lipschitzian condition of BSDE are referred to \cite{Peng1991,Lepeltier1997,Peng1999a,Pardoux1999,Briand2000,Briand2003,Chas2016} and their references therein.  For the study of DPP for stochastic recursive control problem with non-Lipschitzian generator, see \cite{Pu2018,Zhuo2020}.

   In this paper, provided with monotonicity  and polynomial growth  of the generator regarding  part $Y$ of  BSDE, we derive the DPP of the recursive SDG. Since the control processes contain information before the initial time, which differs from the classical stochastic control problems, we adopt the Girsanov transformation approach regarding the directions of the Cameron-Martin space as in \cite{Buckdahn2008}.  A verification theorem by means of classical solution of HJBI equation is also given. Later  based on the DPP, we prove that the lower and upper value functions are respectively the unique  weak solutions  of the related HJBI equations in viscosity sense  by virtue of mollification and truncation techniques (cf. \cite{Pardoux1999}) as well as stability property and comparison theorem of viscosity solution, which generalizes the results in \cite{Wu2008,Buckdahn2008,Pu2018,Zhuo2020} and hence is more applicable in financial markets with non-Lipschitzian preferences  and market friction.
   Moreover, imposing some additional hypotheses and employing  time change of Brownian motion (cf. \cite{Buckdahn2011b}), we improve the classical regularity results of $\frac{1}{2}$-H\"{o}lder continuity on $t$ of value function to Lipschitzian continuity on any subset of $[0,T]$ excluding $T$.

   The main difficulty lies in the existence of nonanticipative strategies in our model, which makes our game problem  not just a  simple generalization from single-control to control-versus-control case, but to a nontrivial control-versus-strategy case. Therefore, the characterization of the optimal control-strategy pair, where control and strategy twine with each other inseparably,  is much harder for us to obtain than the control case in \cite{Pu2018,Zhuo2020}. This  makes two value functions  not  comparable via methods in \cite{Pu2018,Zhuo2020} to prove uniform convergence for a sequence of constructed value functions in Lemma \ref{zhiyizhi} . Utilizing the results in \cite{Buckdahn2008}, we find suitable control-strategy pairs to make two value functions $W$ and $W_m$ comparable in order to obtain the uniformly convergence of value functions in Lemma \ref{zhiyizhi}.

  More importantly, besides the results regarding DPP and viscosity solution, we also obtain two meaningful results. One is   the Lipschitzian continuity (possibly not uniformly in $x$) of value function regarding time variable $t$  on $[0,T-\delta](\delta\in(0,T))$ where we  remove the uniformly bounded requirement of coefficients in \cite{Buckdahn2011b}, which generalizes classical results of $\frac{1}{2}$-H\"{o}lder continuity of  value function regarding $t$. Since the value function involves nonanticipative strategies, which is essentially different from \cite{Buckdahn2011b}, the methods in \cite{Buckdahn2011b} fail to apply. By designing a new nonanticipative strategy $\tilde{\beta}^1$ and through some delicate analysis, we overcome this difficulty, which is nontrivial. This can be useful when considering further numerical computation of the value function or solution of corresponding dynamic-programming equations. The other is  a verification theorem  via  classical solution of  HJBI equation, by virtue of  comparison theorem of BSDEs under monotonic condition, which is applicable for seeking of optimal control-strategy pairs.

This study is organized in this fashion: In Section \ref{sec2}, some notations as well as necessary preliminaries of SDGs are introduced.
In Section \ref{sec3}, the DPP of SDGs with non-Lipschitzian generators and a verification theorem are established in Subsection \ref{sub1} and when the generator does not depend on part $Z$ of BSDE, the lower and upper value functions are proved to be the unique viscosity solutions of the associated HJBI equations in Subsection \ref{sub2}. In Section \ref{sec4},  improved regularity property of value function is obtained. Finally in Section \ref{sec5},  three concrete examples concerned with classical and viscosity solution of HJBI equation as well as financial application are illustrated.

\section{Preliminaries}\label{sec2}
We study  SDGs in the Wiener space $(\Omega, \mathcal{F},P)$, where $\Omega=C_0([0,T];\mathbb{R}^d)$, the space of all continuous functions on $[0,T]$ starting from 0, $\mathcal{F}$ is the complete Borel $\sigma$-algebra regarding the Wiener measure $P$, and $B$ is the coordinate process acting as the Brownian motion: $B_s(\omega)=\omega_s$. Let $\mathcal{N}$ be the $P$-null set in $\mathcal{F}$ and $\mathcal{F}_s:=\sigma\{B_t,t\le s\}\vee\mathcal{N}$.

We then introduce some notations:

For $l\ge1$, and $k,d\in\mathbb{N}^+$,\vspace{1ex}

\noindent$\bullet$\; $L^{2l}(\Omega,\mathcal{F}_T,P;\mathbb{R}^k)$ comprises  $\mathcal{F}_T$-measurable r.v. $\chi\in\mathbb{R}^k$ s.t. $E[|\chi|^{2l}]<+\infty,$\vspace{1ex}

\noindent$\bullet$\; $\mathcal{H}^{2l}_{ d}([0,T])$ comprises  $\{\mathcal{F}_r\}_{0\le r\le T}$-adapted processes $\lambda\in\mathbb{R}^{ d}$ s.t. $E[\normalint_0^T|\lambda_r|^{2l}dr]<+\infty,$\vspace{1ex}

\noindent$\bullet$\; $\mathcal{S}_k^{2l}([0,T])$ comprises  continuous $\{\mathcal{F}_r\}_{0\le r\le T}$-adapted processes $Y\in\mathbb{R}^k$ s.t. $E[\sup\limits_{0\le r\le T}|Y_r|^{2l}]<+\infty$, and $\mathcal{S}^{2l}([0,T]):=\mathcal{S}_k^{2l}([0,T])$, \vspace{1ex}

\noindent$\bullet$\; $\mathbb{S}^k$ comprises   $k\times k$ symmetric matrices,\vspace{1ex}

\noindent$\bullet$\; $C^3_{b}([0,T]\times\mathbb{R}^n)$ comprises $C^3$-functions on $[0,T]\times\mathbb{R}^n$ whose partial
derivatives of order no more than 3 are bounded,\vspace{1ex}

\noindent$\bullet$\; $\mathbb{U},\mathbb{V}$ are compact metric spaces representing control domains,\vspace{1ex}

\noindent$\bullet$\;
$\mathbb{H}:=[0,T]\times\mathbb{R}^n\times\mathbb{R}\times\mathbb{R}^n\times\mathbb{S}^n$, and $\mathbb{G}:=[0,T]\times\mathbb{R}^n\times\mathbb{R}\times\mathbb{R}^d\times\mathbb{U}\times\mathbb{V}$,\vspace{1ex}

\noindent$\bullet$\; $\Xi:=(t,x)$ and $\Xi':=(t,x')$.\vspace{1ex}

For $0\le t<T$,  $\mathcal{U}_{t,T}$ (resp. $\mathcal{V}_{t,T}$) represents all feasible controls for player \Rmnum{1} (resp. \Rmnum{2}) comprising  $\mathbb{U}$ (resp. $\mathbb{V}$)-valued, $\{\mathcal{F}_r\}_{t\le r\le T}$-adapted processes on $[t,T]$. If $u,\overline{u}\in\mathcal{U}_{t,s}$ satisfy $P\{u=\overline{u}, \; a.e., [t,s]\}=1$, then we call $u$ and $\overline{u}$ identical, and denote them as $u\equiv\overline{u}$ on $[t,s]$. Similarly, we denote the identical relationship in $\mathcal{V}_{t,s}$ as $v\equiv\overline{v}$ on $[t,s]$.

For $u\in\mathcal{U}_{t,T}$, $v\in\mathcal{V}_{t,T}$ and  $x\in\mathbb{R}^n$, consider the following stochastic differential equation (SDE) as the state equation
\begin{equation}\label{sdg}
\left\{\begin{array}{ll}
  dx_{s}^{\Xi;u,v}=b(s,x_{s}^{\Xi;u,v},u_s,{{v}_{s}})ds+\sigma (s,x_{s}^{\Xi;u,v},u_s,{{v}_{s}})d{{B}_{s}},\quad s\in [t,T],  \\\\
   x_{t}^{\Xi;u,v}=x.  \\
\end{array} \right.
\end{equation}
Suppose  the coefficients of  SDE \eqref{sdg} fulfill the  hypothesis as follows (all  the Lipschitzian constants of functions in hypotheses of this  paper are denoted by $C>0$):
\begin{hypo}\label{sdexishu}
  \rm{
        $b:[0,T]\times\mathbb{R}^n\times\mathbb{U}\times\mathbb{V}\rightarrow\mathbb{R}^n$ and $\sigma:[0,T]\times\mathbb{R}^n\times\mathbb{U}\times\mathbb{V}\rightarrow\mathbb{R}^{n\times d}$ are continuous regarding $(t,u,v)$ uniformly in $x$, and Lipschitzian regarding $x$ uniformly in $(t,u,v)$.
  }
\end{hypo}
From  Hypothesis \ref{sdexishu},  the linear growth property of $b$ and $\sigma$ regarding $x$, uniformly in $(t,u,v)$ is verified. Hence due to some classical results of SDEs, the unique solvability of SDE \eqref{sdg} in $S^{2l}_n([0,T])$, $\forall l\ge1$, is assured.

Define the cost function
$$J(\Xi;u,v):=Y^{\Xi;u,v}_t,$$
where
\begin{equation}\label{controlled}	Y_{s}^{\Xi;u,v}=h(x_{T}^{\Xi;u,v})+\int_{s}^{T}g(r,x_{r}^{\Xi;u,v},Y_{r}^{\Xi;u,v},Z_{r}^{\Xi;u,v},u_r,v_r)dr-\int_{s}^{T}Z_r^{\Xi;u,v}dB_{r},\quad t\le s\le T.
\end{equation}
Suppose the generator $g:\mathbb{G}\rightarrow\mathbb{R}$ and terminal function $h:\mathbb{R}^n\rightarrow\mathbb{R}$ satisfy the following hypothesis:
\begin{hypo}\label{bsdexishu}
{\rm \quad\begin{enumerate}[(i)]
\item $g$ is continuous regarding  $(t,y,u,v)$ uniformly in $(x,z)$, and Lipschitzian regarding $(x,z)$ uniformly in $(t,y,u,v)$,
\item $h$ is Lipschitzian regarding $x$,
\item there exists a constant $\theta$ s.t. the following relationship holds uniformly in $(\Xi,z,u,v)$:
    $$\forall y,\widetilde{y}
    \in\mathbb{R},\quad (y-\widetilde{y})\cdot\big[g(\cdot,y,\cdot,\cdot,\cdot)-g(\cdot,\widetilde{y},\cdot,\cdot,\cdot)\big]\le\theta|y-\widetilde{y}|^2,$$
\item for a given $p\ge 1$, there exists a positive constant $\eta$ s.t. the following relationship holds uniformly in $(\Xi,y,z,u,v)$:
    $$|g(\cdot,y,\cdot,\cdot,\cdot)-g(\cdot,0,\cdot,\cdot,\cdot)|\le \eta(1+|y|^p).$$
 \end{enumerate}}
   \end{hypo}
   An example of $g$ satisfying Hypothesis \ref{bsdexishu} is the Epstein-Zin preference function presented in Section \ref{sec1}, which will be elaborated in Section \ref{sec5}.
Then   under Hypothesis \ref{bsdexishu}, the unique solvability of BSDE \eqref{controlled} in $\mathcal{S}^{2l}([0,T])\times\mathcal{H}^{2}_{d}([0,T])$, $\forall l\ge1$ is assured (cf. \cite{Pardoux1999}). In fact, we have  estimates as follows:
\begin{prop}\label{yilai}
    Suppose Hypotheses \ref{sdexishu} and \ref{bsdexishu} apply. Then for any $l\ge1$, $u\in\mathcal{U}_{t,T}$, $v\in\mathcal{V}_{t,T}$ and $\xi\in L^{2l}(\Omega, \mathcal{F}_t,P;\mathbb{R}^n)$,  BSDE \eqref{controlled} with $\xi$ as the initial value of SDE \eqref{sdg} admits a unique solution $(Y^{t,\xi;u,v},Z^{t,\xi;u,v})\in\mathcal{S}^{2l}([0,T])\times\mathcal{H}^{2}_{d}([0,T])$. Furthermore, there exists a constant $K(C,\theta,\eta,l,T)>0$, s.t. for arbitrary $\xi,\xi'\in L^{2l}(\Omega,\mathcal{F}_t,P;\mathbb{R}^n)$, it holds that
  $$|Y^{t,\xi;u,v}_t|^{2l}\le K(1+|\xi|^{2l}+E[\normalint_t^T|g(s,0,0,0,u_s,v_s)|^{2l}ds|\mathcal{F}_t]),$$
  $$|Y^{t,\xi;u,v}_t-Y^{t,\xi';u,v}_t|\le K|\xi-\xi'|,$$
  and
  $$E[\sup\limits_{t\le r\le T}|Y^{t,\xi;u,v}_r|^2+\int_t^T|Y^{t,\xi;u,v}_r|^{2l-2}|Z^{t,\xi;u,v}_r|dr]\le K(1+E[|\xi|^{2l}+\int_t^T|g(r,0,0,0,u_r,v_r)|^{2l}dr]).$$

\end{prop}
The proof is standard, where  monotonic condition (iii) in Hypothesis \ref{bsdexishu} is employed instead of  Lipschitzian condition when applying It\^{o}'s formula, so we omit it.
Notice that since $u_r$, $v_r$ are $\{\mathcal{F}_r\}_{t\le r\le T}$-adapted, i.e., containing the information before the initial time $t$, $J(\Xi;u,v)$ is not necessarily deterministic  but $\mathcal{F}_t$-measurable.

Finally we introduce the notion of  feasible strategies.
\begin{definition}
A mapping $\alpha$ from $\mathcal{V}_{t,T}$ to $\mathcal{U}_{t,T}$ is called a nonanticipative strategy for player \Rmnum{1}, if for every $[t,T]$-valued, $\mathcal{F}_r$-stopping time $S$ and $v_1, v_2\in\mathcal{V}_{t,T}$ s.t. $v_1\equiv v_2$ on $[t,S]$, we have $\alpha[v_1]\equiv\alpha[v_2]$ on $[t,S]$. Similarly, we can define a nonanticipative strategy on $[t,T]$ $\beta:\mathcal{U}_{t,T}\rightarrow\mathcal{V}_{t,T}$ for player \Rmnum{2}. Let $\mathcal{A}_{t,T}$ (resp. $\mathcal{B}_{t,T}$) represent  all nonanticipative strategies for player \Rmnum{1} (resp. player \Rmnum{2}) on $[t,T]$.
\end{definition}
Then we can divide the  zero-sum SDGs into two types:

\noindent(i) lower  games where  lower value functions are
 \begin{equation}\label{lower}
  W(\Xi):=\mathop{\essinf}\limits_{\beta\in\mathcal{B}_{t,T}}\mathop{\esssup}\limits_{u\in\mathcal{U}_{t,T}}J(\Xi;u,\beta(u)),
\end{equation}
(ii) upper games where  upper value functions are
\begin{equation}\label{upper}
  U(\Xi):=\mathop{\esssup}\limits_{\alpha\in\mathcal{A}_{t,T}}\mathop{\essinf}\limits_{v\in\mathcal{V}_{t,T}}J(\Xi;\alpha(v),v).
\end{equation} A control-strategy pair $(u,\beta)$ for a lower game is called an optimal control-strategy pair, if the cost functional attains value function $W$  under $(u,\beta)$. For upper game we give symmetric definition.
\begin{remark}
  Notice that in lower and upper games, one player gets to know the choice of the other player's control  to make an informed decision, the SDGs are actually of information asymmetric condition, which are applicable in principal-agent problems, see Cvitani{\'{c}} et al. \cite{Cvitanic2013a}.
\end{remark}
The names of ``upper value function'' and ``lower value function'' of $U$ and $W$ are justified in Remark \ref{daxiao}.
 Since the lower games and upper games  are symmetrical, we only consider the lower case for  theoretical derivation.
\section{ DPP and  viscosity solution to HJBI equation}\label{sec3}
This section is divided into two parts. Subsection \ref{sub1} is dedicated to the determinicity and  (improved) regularity of the lower value function $W$, DPP of the lower SDG, and a verification theorem, while Subsection \ref{sub2} illustrates the unique weak (viscosity) solvability of    HJBI equation \eqref{hjbi} when $g$ is independent of $Z$.
\subsection{ DPP of  SDGs with non-Lipschitzian generators and regularity of value function}\label{sub1}

Since $J(\Xi;u,v)$ is $\mathcal{F}_t$-measurable, $W$ is $\mathcal{F}_t$-measurable. Nonetheless
by showing that $W$ is invariant under Girsanov transformation regarding all directions of Cameron-Martin space, i.e., the space of all differentiable functions in $\Omega$, whose derivatives are integrable on $[0,T]$,  we can even assert that $W$ is  deterministic.
\begin{lemma}
    $W$ is deterministic.
\end{lemma}
 Since $(x^{\Xi;u,v},Y^{\Xi;u,v},Z^{\Xi;u,v})\in S_n^{2l}([0,T])\times S^{2l}([0,T])\times H^2_d([0,T])$, $\forall l\ge 1$, the Girsanov transformation of Proposition 3.3 in \cite{Buckdahn2008} still holds here. Thus the standard arguments of Proposition 3.3 in \cite{Buckdahn2008} still apply here to get our result.

Then we   show the continuity and linear growth of $W$ regarding $x$:
\begin{prop}\label{xlianxu}
 Suppose Hypotheses \ref{sdexishu} and \ref{bsdexishu} apply, then for arbitrary $t\in[0,T]$ and $x,x'\in\mathbb{R}^n$, there exists a constant $K'$ s.t.
\begin{enumerate}[(i)]
\item $|W(\Xi)-W(\Xi')|\le K'|x-x'|,$
\item $|W|\le K'(1+|x|).$
\end{enumerate}
\end{prop}
\begin{proof}
By the definition of $W$, we have
\begin{align*}
  W&\le \mathop{\esssup}\limits_{\beta\in\mathcal{B}_{t,T}}\mathop{\esssup}\limits_{u\in\mathcal{U}_{t,T}}J(\Xi;u,\beta(u))\\
  &\le\mathop{\esssup}\limits_{u\in\mathcal{U}_{t,T},v\in\mathcal{V}_{t,T}}J(\Xi;u,v)\\
&\le\mathop{\esssup}\limits_{u\in\mathcal{U}_{t,T},v\in\mathcal{V}_{t,T}}|J(\Xi;u,v)|\\
&\overset{(a)}\le\mathop{\esssup}\limits_{u\in\mathcal{U}_{t,T},v\in\mathcal{V}_{t,T}}K'(1+|x|)\\
&=K'(1+|x|), \quad P-a.s.,
\end{align*}
where (a) follows from Proposition \ref{yilai} and the compactness of $\mathbb{U}$ and $\mathbb{V}$. Similarly we have $-W\le K'(1+|x|)$. Thus (ii) is proved.

For $(x,x')\in\mathbb{R}^n$,
\begin{align*}
     W(\Xi)-W(\Xi')
    &\le \mathop{\esssup}\limits_{\beta\in\mathcal{B}_{t,T}}|\mathop{\esssup}\limits_{u\in\mathcal{U}_{t,T}}J(\Xi;u,\beta(u))-\mathop{\esssup}\limits_{u\in\mathcal{U}_{t,T}}J(\Xi';u,\beta(u))|\\
    &\le \mathop{\esssup}\limits_{\beta\in\mathcal{B}_{t,T}}\mathop{\esssup}\limits_{u\in\mathcal{U}_{t,T}}|J(\Xi;u,\beta(u))-J(\Xi';u,\beta(u))|\\
    &\le \mathop{\esssup}\limits_{u\in\mathcal{U}_{t,T},v\in\mathcal{V}_{t,T}}|J(\Xi;u,v)-J(\Xi';u,v)|\\
    &\overset{(b)}\le K'|x-x'|, \quad P-a.s.,
  \end{align*}
where (b) is due to Proposition \ref{yilai}. Similarly we have $W(\Xi')-W(\Xi)\le K'|x-x'|$. Thus (i) is proved.
\end{proof}
We now give the DPP of the lower SDG.
\begin{theorem}\label{dpp} Suppose Hypotheses \ref{sdexishu} and \ref{bsdexishu} apply, then it holds that
 $$W=\underset{\beta \in {{\mathcal{B}}_{t,t+\delta }}}{\mathop{\essinf}}\,\underset{u\in {{\mathcal{U}}_{t,t+\delta }}}{\mathop{\esssup }}\,G_{t,t+\delta }^{\Xi;u,\beta (u)}[W(t+\delta ,x_{t+\delta }^{\Xi;u,\beta (u)})],\quad 0\le t<t+\delta \le T,$$
where
$$G_{s,t+\delta }^{\Xi;u,v}[\eta ]:=\widetilde{Y}_{s}^{\Xi;u,v},\quad s\in [t,t+\delta ],$$
$$\widetilde{Y}_{s}^{\Xi;u,v}=\eta +\int_{s}^{t+\delta }g(r,x_{r}^{\Xi;u,v},\widetilde{Y}_{r}^{\Xi;u,v},\widetilde{Z}_{r}^{\Xi;u,v},{{u}_{r}},{{v}_{r}})dr-\int_{s}^{t+\delta }{\widetilde{Z}_{r}^{\Xi;u,v}}d{{B}_{r}}.$$
\end{theorem}
\begin{proof}Owing to Propositions \ref{yilai} and \ref{xlianxu}, the same continuity dependence property of  cost function BSDE \eqref{controlled} as well as continuity and linear growth properties of $W$ still apply here. Moreover, the  comparison theorem for BSDEs with non-Lipschitzian generators (e.g., comparison theorem in \cite{Fan2012}) still holds  for the cost function and the value function.
Then employing   standard arguments of Theorem 3.6 in \cite{Buckdahn2008} gets our result.
\end{proof}
\begin{remark}\label{queding} Analogues to Remark 3.4 in \cite{Buckdahn2008} we have under Hypotheses \ref{sdexishu} and \ref{bsdexishu}, that

\noindent{\rm (a)}
$$W=\inf\limits_{\beta \in \mathcal{B}_{t,T}}\sup\limits_{u\in {{\mathcal{U}}_{t,T}}}E[J(\Xi;u,\beta (u))],$$
and
$$U=\sup\limits_{\alpha\in\mathcal{A}_{t,T}}\inf\limits_{v\in\mathcal{V}_{t,T}}  E[J(\Xi;\alpha(v),v)],$$
{\rm (b)}  for any $\delta\in(0,T-t]$ and $\varepsilon>0$ the following hold:

\noindent for any $\beta\in\mathcal{B}_{t,t+\delta}$, there exists some $u^\varepsilon\in\mathcal{U}_{t,t+\delta}$, s.t.
$$W\le G^{\Xi;u^\varepsilon,\beta(u^\varepsilon)}_{t,t+\delta}[W(t+\delta,x^{\Xi;u^\varepsilon,\beta(u^\varepsilon)}_{t+\delta})]+\varepsilon,\quad P-a.s.,$$
and there exists some $\beta^\varepsilon\in\mathcal{B}_{t,t+\delta}$ s.t., for all $u\in\mathcal{U}_{t,t+\delta}$,
$$W\ge G^{\Xi;u,\beta^\varepsilon(u)}_{t,t+\delta}[W(t+\delta,x^{\Xi;u,\beta^{\varepsilon}(u)}_{t+\delta})]-\varepsilon,\quad P-a.s.$$
\end{remark}
Subsequently we also get the $\frac{1}{2}$-H\"{o}lder continuity of $W$ regarding $t$ by standard arguments of Theorem 3.10 in \cite{Buckdahn2008} along with Theorem \ref{dpp}.
\begin{prop}\label{tlianxu}
 Suppose Hypotheses \ref{sdexishu} and \ref{bsdexishu} apply, then $W$ is $\frac{1}{2}$-H\"{o}lder continuous regarding $t$.
\end{prop}

 We then give a verification theorem where coefficients depend on $(u,v)$, which also contains results for $U$ for reader's convenience.
\begin{theorem}Suppose Hypotheses \ref{sdexishu} and \ref{bsdexishu} apply.

 (i) Suppose there exists a pair of mappings $\psi_2$ and $\phi_1$  s.t.
  \begin{align*}
  &  \quad \psi_2(\Xi,u,y,p,A)\in{\rm{argmin}\ } \mathcal{H}^{Z} (\Xi,y,u,\cdot,p,A)\\
  & \equiv\left\{\bar{v}\in \mathbb{V}|\mathcal{H}^{Z}(\Xi,y,u,\bar{v},p,A)=\underset{v\in\mathbb{V}}{\min}\ \mathcal{H}^Z(\Xi,y,u,v,p,A)\right\},
  \end{align*}
  \begin{align*}
  &  \phi_1(\Xi,y,p,A)\in{\rm{argmax}\ } \mathcal{H}^{Z} (\Xi,y,\cdot,\psi_2(\Xi,\cdot,y,p,A),p,A)\\
  & \equiv\left\{\bar{u}\in \mathbb{U}|\mathcal{H}^{Z}(\Xi,y,\bar{u},\psi_2(\Xi,\bar{u},y,p,A),p,A)=\underset{u\in\mathbb{U}}{\max}\ \mathcal{H}^Z(\Xi,y,u,\psi_2(\Xi,u,y,p,A),p,A)\right\},
  \end{align*}
  where
   \begin{align}&\mathcal{H}_{{}}^{Z}(\Xi,y,u,v,p,A)
 =\frac{1}{2}\Tr(\sigma {{\sigma }^{T}}(\Xi,u,v)A)+ \langle p,b(\Xi,u,v)\rangle+g(\Xi,y,p^{T}\sigma(\Xi,u,v),u,v). \nonumber
\end{align}
Suppose $W$ is a classical solution to HJBI equation \eqref{hjbiz}, where $|W|+|\partial_t W|+|D W|+|D^2 W|\le C(1+|x|)$, $\forall t\in[0,T]$.
\begin{equation}\label{hjbiz}
  \left\{\begin{array}{ll}
  \partial _{t}W+H^{-}(\Xi,W,DW,D^{2}W)=0,\quad    (\Xi)\in [0,T)\times \mathbb{R}^n,  \\\\
   W|_{t=T}=h(x),\quad x\in \mathbb{R}^n,
  \end{array}\right.
\end{equation}
\begin{align}&H_{{}}^{-}(\Xi,y,p,A)
 =\underset{u\in \mathbb{U}}{\mathop{\sup }}\,\underset{v\in \mathbb{V}}{\mathop{\inf }}\,\{\mathcal{H}_{{}}^{Z}(\Xi,y,u,v,p,A)\}. \nonumber
\end{align}
  Define \begin{align*}
&    \hat{u}(\Xi)=\phi_1(\Xi,W(\Xi),D W(\Xi), D^2 W(\Xi)),\\
& \hat{\beta}(\Xi,u)=\psi_2(\Xi,u,W(\Xi),D W(\Xi), D^2 W(\Xi)),
  \end{align*}which are assumed to be Lipschitzian regarding $(\Xi)$ and $(\Xi, u)$, respectively.  Then $\hat{u}(\Xi)$, $\hat{\beta}(\Xi,u)$ is the optimal feedback control-strategy pair and $W(\Xi)$ is the lower value function, i.e., $$W(\Xi)=J(\Xi;\hat{u},\hat{\beta})=\mathop{\essinf}\limits_{\beta\in\mathcal{B}_{t,T}}\mathop{\esssup}\limits_{u\in\mathcal{U}_{t,T}}J(\Xi;u,\beta(u)).$$    Moreover, $\hat{u}(\Xi)$, $\hat{\beta}(\Xi,u)$ is a saddle point in the sense that
  $$J(\Xi;\hat{u},\hat{\beta})=\underset{v\in\mathcal{V}_{t,T}}{\rm{essinf}}\ J(\Xi;\hat{u},v)=\underset{u\in\mathcal{U}_{t,T}}{\rm{esssup}}\ J(\Xi;u,\hat{\beta}(\Xi,u)).$$

 (ii) Suppose there exists a pair of mappings $\psi_1$ and $\phi_2$  s.t.
 \begin{align*}
  &  \psi_1(\Xi,v,y,p,A)\in{\rm{argmax}\ } \mathcal{H}^{Z} (\Xi,y,\cdot,v,p,A),\\
  &  \phi_2(\Xi,y,p,A)\in{\rm{argmin}\ } \mathcal{H}^{Z} (\Xi,y,\psi_1(\Xi,\cdot,y,p,A),\cdot,p,A),
  \end{align*}
  \begin{align*}
&    \hat{v}(\Xi)=\phi_2(\Xi,U(\Xi),DU(\Xi),D^2U(\Xi)),\\
 &   \hat{\alpha}(\Xi,v)=\psi_1(\Xi,v,U(\Xi),DU(\Xi),D^2U(\Xi)),
  \end{align*}
  are Lipschitzian regarding $(\Xi)$ and $(\Xi, v)$, respectively. Moreover, suppose $U$ is a classical solution to HJBI equation \eqref{upperhjbiz}  where $|U|+|\partial_t U|+ |D U|+|D^2 U|\le C(1+|x|)$, $\forall t\in[0,T]$,
  \begin{equation}\label{upperhjbiz}
  \left\{\begin{array}{ll}
  \partial _{t}U+H^{+}(\Xi,U,DU,D^{2}U)=0,\quad    (\Xi)\in [0,T)\times \mathbb{R}^n,  \\\\
   U|_{t=T}=h(x),\quad x\in \mathbb{R}^n,
  \end{array}\right.
\end{equation}
\begin{align}&H_{{}}^{+}(\Xi,y,p,A)=\underset{v\in \mathbb{V}}{\mathop{\inf }}\,\underset{u\in \mathbb{U}}{\mathop{\sup }}\,\{\mathcal{H}_{{}}^{Z}(\Xi,y,u,v,p,A)\}.\nonumber
\end{align}
 Then $\hat{v}(\Xi)$, $\hat{\alpha}(\Xi,v)$ is the optimal feedback control-strategy pair and $U(\Xi)$ is the lower value function.    Moreover, $\hat{v}(\Xi)$, $\hat{\alpha}(\Xi,v)$ is a saddle point in the sense that
 $$J(\Xi;\hat{\alpha},\hat{v})=\underset{v\in\mathcal{V}_{t,T}}{\rm{essinf}}\ J(\Xi;\hat{\alpha}(\Xi,v),v)=\underset{u\in\mathcal{U}_{t,T}}{\rm{esssup}}\ J(\Xi;u,\hat{v}).$$ \end{theorem}
\begin{proof}
 Since declarations (i) and (ii) are symmetric, we only give a proof of  (i).

 Fix $\hat{\beta}$. For $u\in\mathcal{U}_{t,T}$, denote $\check{x}$ and $\check{Y}$ the solution of SDE \eqref{sdg} and BSDE \eqref{controlled} with controls $u$ and $\check{v}:=\hat{\beta}(\cdot,\check{x}(\cdot),u(\cdot))$.  It\^{o}'s formula gives
\begin{align*} W(s,\check{x}_s)
=&\ h(\check{x}_T)-\int_s^T\big[\partial_rW(r,\check{x}_r)\\
& +\frac{1}{2}\Tr ( \sigma\sigma^{T}(r,\check{x}_r,u_r,\check{v}_r)D^2W(r,\check{x}_r))+\langle D W(r,\check{x}_r),b(r,\check{x}_r,u_r,\check{v}_r)\rangle \big]dr\\
& -\int_s^T(DW)^T(r,\check{x}_r)\sigma(r,\check{x}_r,u_r,\check{v}_r) dB_r.\end{align*}
Plugging  $\mathcal{H}^Z$ into the equation above, we obtain
\begin{align*} W(s,\check{x}_s)
=&\ h(\check{x}_T)\\
& +\int_s^T\big[g(r,\check{x}_r,W(r,\check{x}_r),(DW)^T(r,\check{x}_r)\sigma(r,\check{x}_r,u_r,\check{v}_r) ,u_r,\check{v}_r)-\partial_rW(r,\check{x}_r)\\
&-\mathcal{H}^Z(r,\check{x}_r,W(r,\check{x}_r),u_r,\check{v}_r,D W(r,\check{x}_r),D^2 W(r,\check{x}_r))\big]dr\\
&-\int_s^T(DW)^T(r,\check{x}_r)\sigma(r,\check{x}_r,u_r,\check{v}_r) dB_r.\end{align*}
By means of $\hat{\beta}$ and HJBI equation \eqref{hjbiz},
we have
\begin{align}
   & \nonumber \quad-\partial_rW(r,\check{x}_r)-\mathcal{H}^Z(r,\check{x}_r,W(r,\check{x}_r),u_r,\check{v}_r,D W(r,\check{x}_r),D^2 W(r,\check{x}_r) )\\
   & \nonumber=-\partial_r W(r,\check{x}_r)-\underset{v\in\mathbb{V}}{\inf}\ \mathcal{H}^Z(r,\check{x}_r,W(r,\check{x}_r),u_r,v,D W(r,\check{x}_r),D^2 W(r,\check{x}_r))\\
   &\ge -\partial_r W(r,\check{x}_r)-\underset{u\in\mathbb{U}}{\sup}\  \underset{v\in\mathbb{V}}{\inf}\ \mathcal{H}^Z(r,\check{x}_r,W(r,\check{x}_r),u,v,D W(r,\check{x}_r),D^2 W(r,\check{x}_r))=0.\label{le}
\end{align}
Due to comparison theorem for BSDEs with non-Lipschitzian generators (e.g., comparison theorem in \cite{Fan2012}), $W(\Xi)\ge \check{Y}_t$. Replacing $u$ by $\hat{u}$, the inequality of  \eqref{le} becomes an equality, and hence $W(\Xi)=J(\Xi;\hat{u},\hat{\beta})=\underset{u\in\mathcal{U}_{t,T}}{\mathop{\esssup }}\ J(\Xi;u,\hat{\beta})$. Therefore, $(\hat{u},\hat{\beta})$ is an optimal control-strategy pair.

Fix $\hat{u}$, for any $v\in\mathcal{V}_{t,T}$, similarly we have
$W(\Xi)=\underset{v\in\mathcal{V}_{t,T}}{\mathop{\essinf }}\ J(\Xi;\hat{u},v)$. Then $(\hat{u},\hat{\beta})$ is a saddle point.
\end{proof}
\subsection{Viscosity solution to HJBI equation}\label{sub2}
Since we can only assert the continuity of $W$ by Propositions \ref{xlianxu}, \ref{tlianxu} and \ref{yizhilip} under the monotonic condition,  the derived HJBI equation  might not admit a classical solution, by which we can only consider  viscosity solutions.
 For this purpose, we investigate the case where $g$ is free of $z$, i.e., $g(\Xi,y,z,u,v)=g(\Xi,y,u,v)$ in BSDE \eqref{controlled}. The HJBI equation regarding $W$ is
\begin{equation}\label{hjbi}
  \left\{\begin{array}{ll}
  \partial _{t}W+H^{-}(\Xi,W,DW,D^{2}W)=0,\quad    (\Xi)\in [0,T)\times \mathbb{R}^n,  \\\\
   W|_{t=T}=h(x),\quad x\in \mathbb{R}^n,
  \end{array}\right.
\end{equation}
where
\begin{align}&H_{{}}^{-}(\Xi,y,p,A)\nonumber\\
 =\:\;&\underset{u\in \mathbb{U}}{\mathop{\sup }}\,\underset{v\in \mathbb{V}}{\mathop{\inf }}\,\{\frac{1}{2}\Tr(\sigma {{\sigma }^{T}}(\Xi,u,v)A)+ \langle p, b(\Xi,u,v)\rangle+g(\Xi,y,u,v)\}, \nonumber
\end{align}
and regarding $U$ is
 \begin{equation}\label{upperhjbi}
  \left\{\begin{array}{ll}
  \partial _{t}U+H^{+}(\Xi,U,DU,D^{2}U)=0,\quad    (\Xi)\in [0,T)\times \mathbb{R}^n,  \\\\
   U|_{t=T}=h(x),\quad x\in \mathbb{R}^n,
  \end{array}\right.
\end{equation}
where
\begin{align}&H_{{}}^{+}(\Xi,y,p,A)\nonumber\\
 =\:\;&\underset{v\in \mathbb{V}}{\mathop{\inf }}\,\underset{u\in \mathbb{U}}{\mathop{\sup }}\,\{\frac{1}{2}\Tr(\sigma {{\sigma }^{T}}(\Xi,u,v)A)+ \langle p, b(\Xi,u,v)\rangle+g(\Xi,y,u,v)\},\nonumber
\end{align}
for
$(\Xi,y,p,A)\in\mathbb{H}.$

We give the definition of viscosity solutions of HJBI equation \eqref{hjbi}. Again, only the lower case is considered for theoretical derivation.
\begin{definition}\label{nxjdingyi}
  A continuous function $W:[0,T]\times\mathbb{R}^n\rightarrow \mathbb{R}$ is called

  \noindent{\rm(i)} a viscosity subsolution to HJBI equation \eqref{hjbi}
 if $W(T,x)\le h(x)$ for arbitrary $x\in\mathbb{R}^n$ and if for arbitrary function $\phi\in C^3_{b}([0,T]\times\mathbb{R}^n)$ and $(\Xi)\in[0,T]\times\mathbb{R}^n$ s.t. $W-\phi$ attains its local maximum at $(\Xi)$:
$$\partial_t\phi+H^-(\Xi,\phi,D\phi,D^2\phi)\ge 0,$$
{\rm(ii)} a viscosity  supersolution to HJBI equation \eqref{hjbi}
 if $W(T,x)\ge h(x)$ for arbitrary $x\in\mathbb{R}^n$ and if for arbitrary function $\phi\in C^3_{b}([0,T]\times\mathbb{R}^n)$ and $(\Xi)\in[0,T]\times\mathbb{R}^n$ s.t. $W-\phi$ attains its local minimum at $(\Xi)$:
$$\partial_t\phi+H^-(\Xi,\phi,D\phi,D^2\phi)\le 0,$$
{\rm(iii)} a viscosity solution to HJBI equation \eqref{hjbi} if it is both a viscosity subsolution and a viscosity supersolution to \eqref{hjbi}.
\end{definition}
We first construct a sequence of $g$'s mollifications  regarding $y$:
\begin{align*}
 & \quad g_m(\Xi,y,u,v)\\&:=(\zeta_m(\cdot) * g(\Xi,\cdot,u,v))(y)=\int_\mathbb{R}\zeta_m(y-y') g(\Xi,y',u,v)dy'=\int_\mathbb{R}\zeta_m(y') g(\Xi,y-y',u,v)dy',
\end{align*}
where $\zeta_m:\mathbb{R}\rightarrow\mathbb{R}^+\in C_c^{\infty}(\mathbb{R})$ is a sequence of smooth mollifiers    with compact support in $[-\frac{1}{m},\frac{1}{m}]$ satisfying $\int_\mathbb{R}\zeta_m(y)dy=1.$ Note that  $\zeta_m\in C_c^{\infty}(\mathbb{R})$ implies that all the derivatives of $\zeta_m$ in $\mathbb{R}$ are bounded, hence $\zeta_m$ is also a Lipschitzian function in $\mathbb{R}$.

Then we could easily verify that $g_m$ fulfills Hypothesis \ref{bsdexishu}.

Introduce the BSDE with  generator $g_m$:
\begin{equation}\label{moguangbsde}	Y_{s}^{\Xi,m;u,v}=h(x_{T}^{\Xi;u,v})+\int_{s}^{T}g_m(r,x_{r}^{\Xi;u,v},Y_{r}^{\Xi,m;u,v},u_r,v_r)dr-\int_{s}^{T}Z_r^{\Xi,m;u,v}dB_{r},\quad t\le s\le T.
\end{equation}
Define the associated cost function $J_m(\Xi;u,v):=Y^{\Xi,m;u,v}_t$,    lower value function
\begin{equation}\label{moguangzhi}W_m:=\underset{\beta \in {{\mathcal{B}}_{t,T}}}{\mathop{\essinf }}\,\underset{u\in {{\mathcal{U}}_{t,T}}}{\mathop{\esssup }}\,J_m(\Xi;u,\beta (u)),
\end{equation}
and  Hamiltonian
\begin{align}\label{moguangha}&H_{m}^{-}(\Xi,y,p,A)\nonumber\\
 =\:\;&\underset{u\in \mathbb{U}}{\mathop{\sup }}\,\underset{v\in \mathbb{V}}{\mathop{\inf }}\,\{\frac{1}{2}\Tr(\sigma {{\sigma }^{T}}(\Xi,u,v)A)+\langle p, b(\Xi,u,v)\rangle+g_m(\Xi,y,u,v)\},
\end{align}
for $(\Xi,y,p,A)\in\mathbb{H}.$

Subsequently we present a series of Lemmas to show the uniform convergence of $g_m$, $J_m$, $W_m$ and $H^-_m$  to $g$, $J$, $W$ and $H^-$, in   arbitrary compact set of their   domains.
\begin{lemma}\label{xishuyizhi}
  Suppose (i) in Hypothesis \ref{bsdexishu} applies, then $g_m\rightarrow g$ as $m\rightarrow\infty$, uniformly in any compact set  $\mathcal{O}\subset[0,T]\times\mathbb{R}^n\times\mathbb{R}\times\mathbb{U}\times\mathbb{V}$.
\end{lemma}
\begin{proof}
Due to the fact that $\int_\mathbb{R}\zeta_m(y')dy'=1$,
$$g_m(\Xi,y,u,v)-g(\Xi,y,u,v)=\int_\mathbb{R}(g(\Xi,y-y',u,v)-g(\Xi,y,u,v))\zeta_m(y')dy'.$$
For any compact set $\mathcal{O}\in[0,T]\times\mathbb{R}^n\times\mathbb{R}\times\mathbb{U}\times\mathbb{V}$, there exists another compact set $\widetilde{\mathcal{O}}$ s.t. $(\Xi,y-y',u,v)\in\widetilde{\mathcal{O}}$ for any $(\Xi,y,u,v)\in\mathcal{O}$ and $y'\in[-1,1]$. By virtue of the continuity regarding $(t,y,u,v)$ and Lipschitzian continuity regarding $x$ of $g$, we get the uniform continuity of $g$ regarding $(\Xi,y,u,v)$ in $\widetilde{\mathcal{O}}$, which means for arbitrary $\varepsilon>0$ and $m$ sufficiently large, we get
\begin{align*}
  &\quad\sup\limits_{(\Xi,y,u,v)\in\mathcal{O}}|g_m(\Xi,y,u,v)-g(\Xi,y,u,v)|\\&\le \sup\limits_{(\Xi,y,u,v)\in\mathcal{O}}\int_{[-\frac{1}{m},\frac{1}{m}]}|g(\Xi,y-y',u,v)-g(\Xi,y,u,v)|\zeta_m(y')dy'\\
&\le\varepsilon\int_{[-\frac{1}{m},\frac{1}{m}]}\zeta_m(y')dy'=\varepsilon,
\end{align*}
which finishes the proof.
\end{proof}
The verification procedures of the following Lemma  are standard, analogues to Lemma 4.4 of \cite{Pu2018}, thus are skipped.
\begin{lemma}\label{costyizhi}
  Suppose Hypotheses \ref{sdexishu} and \ref{bsdexishu} apply, then for any $u\in\mathcal{U}_{t,T}$, $v\in\mathcal{V}_{t,T}$, and any compact set $\mathcal{O}\subset[0,T]\times\mathbb{R}^n$,
  $$\sup\limits_{(\Xi)\in \mathcal{O}}E[|Y^{\Xi,m;u,v}_t-Y^{\Xi;u,v}_t|^2]\rightarrow 0, \quad m\rightarrow\infty. $$ Moreover, due to the compactness of control domain, the convergence is uniformly in $(u,v)\in \mathcal{U}_{t,T}\times\mathcal{V}_{t,T}$.
\end{lemma}
\begin{lemma}\label{zhiyizhi}
  Suppose Hypotheses \ref{sdexishu} and \ref{bsdexishu} apply, then $W_m\rightarrow W$, $m\rightarrow\infty$, uniformly in any compact set $\mathcal{O}\in[0,T]\times\mathbb{R}^n$.
\end{lemma}
\begin{proof}
   By means of Remark \ref{queding}, we have for any $\varepsilon>0$:

   \noindent(a) There exist some $\beta_m^{\varepsilon},\beta^\varepsilon\in\mathcal{B}_{t,T}$ s.t. for any $u\in\mathcal{U}_{t,T}$,
  \begin{align}
     W_m&\ge  J_m(\Xi;u,\beta_m^\varepsilon(u))-\varepsilon,\quad P-a.s.,\label{1}\\
     W&\ge J(\Xi;u,\beta^\varepsilon(u))-\varepsilon,\quad P-a.s.;\label{2}
  \end{align}
(b) for any $\beta\in\mathcal{B}_{t,T}$, there exist $u_m^\varepsilon,u^\varepsilon\in\mathcal{U}_{t,T}$ s.t.
\begin{align}
  W_m&\le J_m(\Xi;u_m^\varepsilon,\beta(u_m^\varepsilon))+\varepsilon,\quad P-a.s.,\label{3}\\
  W&\le J(\Xi;u^\varepsilon,\beta(u^\varepsilon))+\varepsilon,\quad P-a.s.\label{4}
\end{align}
Let $\beta=\beta^\varepsilon$ in \eqref{3} and $u=u_m^\varepsilon$ in \eqref{2}, then we get
\begin{equation}\label{5}
  W-W_m\ge J(\Xi;u^\varepsilon_m,\beta^\varepsilon(u^\varepsilon_m))-J_m(\Xi;u^\varepsilon_m,\beta^\varepsilon(u^\varepsilon_m))-2\varepsilon.
\end{equation}
Let $\beta=\beta^\varepsilon_m$ in \eqref{4} and $u=u^\varepsilon$ in \eqref{1}, then we get
\begin{equation}\label{6}
  W-W_m\le J(\Xi;u^\varepsilon,\beta_m^\varepsilon(u^\varepsilon))-J_m(\Xi;u^\varepsilon,\beta_m^\varepsilon(u^\varepsilon))+2\varepsilon.
\end{equation}
Combining \eqref{5} and \eqref{6}, we have
\begin{align}
  &\quad |W-W_m|-4\varepsilon\nonumber\\
  &\le E[|J(\Xi;u^\varepsilon_m,\beta^\varepsilon(u^\varepsilon_m))-J_m(\Xi;u^\varepsilon_m,\beta^\varepsilon(u^\varepsilon_m))|]+E[|J(\Xi;u^\varepsilon,\beta_m^\varepsilon(u^\varepsilon))-J_m(\Xi;u^\varepsilon,\beta_m^\varepsilon(u^\varepsilon))|].
\end{align}
Since $u^\varepsilon_m,\beta^\varepsilon(u^\varepsilon_m)$ and $u^\varepsilon_m,\beta_m^\varepsilon(u^\varepsilon)$ are feasible controls, by Lemma \ref{costyizhi}, for any compact set $\mathcal{O}\subset[0,T]\times\mathbb{R}^n$,
  $$\sup\limits_{(\Xi)\in \mathcal{O}}E[|J(\Xi;u^\varepsilon_m,\beta^\varepsilon(u^\varepsilon_m))-J_m(\Xi;u^\varepsilon_m,\beta^\varepsilon(u^\varepsilon_m))|+|J(\Xi;u^\varepsilon,\beta_m^\varepsilon(u^\varepsilon))-J_m(\Xi;u^\varepsilon,\beta_m^\varepsilon(u^\varepsilon))|]\rightarrow 0,   $$
  $m\rightarrow\infty$. Due to the arbitrariness of $\varepsilon$, we get our result.
\end{proof}
\begin{lemma}\label{hayizhi}
  Suppose Hypotheses \ref{sdexishu} and \ref{bsdexishu} apply. Then $H^-_m\rightarrow H^-$, uniformly in any compact set of their domain.
\end{lemma}
\begin{proof}
  Let $$\mathcal{H}_m=\frac{1}{2}\Tr(\sigma {{\sigma }^{T}}(\Xi,u,v)A)+\langle p,b(\Xi,u,v)\rangle+g_m(\Xi,y,u,v)$$
  and
  $$\mathcal{H}=\frac{1}{2}\Tr(\sigma {{\sigma }^{T}}(\Xi,u,v)A)+\langle p,b(\Xi,u,v)\rangle+g(\Xi,y,u,v).$$
  Then\begin{align*}
     H_m^--H^-
    &=\sup\limits_{u\in\mathbb{U}}\inf\limits_{v\in\mathbb{V}}\mathcal{H}_m-\sup\limits_{u\in\mathbb{U}}\inf\limits_{v\in\mathbb{V}}\mathcal{H}\\
    &\le \sup\limits_{u\in\mathbb{U}}(\inf\limits_{v\in\mathbb{V}}\mathcal{H}_m-\inf\limits_{v\in\mathbb{V}}\mathcal{H})\\
    &\le \sup\limits_{u\in\mathbb{U}}\sup\limits_{v\in\mathbb{V}}(\mathcal{H}_m-\mathcal{H})\\
    &\le \sup\limits_{u\in\mathbb{U},v\in\mathbb{V}}|\mathcal{H}_m-\mathcal{H}|\\
    &=\sup\limits_{u\in\mathbb{U},v\in\mathbb{V}}|g_m-g|
  \end{align*}
  and\begin{align*}
    H^--H_m^-
    &=\sup\limits_{u\in\mathbb{U}}\inf\limits_{v\in\mathbb{V}}\mathcal{H}-\sup\limits_{u\in\mathbb{U}}\inf\limits_{v\in\mathbb{V}}\mathcal{H}_m\\
    &\le\sup\limits_{u\in\mathbb{U},v\in\mathbb{V}}|g_m-g|.
  \end{align*}
 Thus
 $$|H^--H_m^-|\le \sup\limits_{u\in\mathbb{U},v\in\mathbb{V}}|g_m-g|.$$
 For arbitrary  $\mathcal{O}$ compact in $\mathbb{H}$ and $(\Xi,y,p,A,u,v)\in \mathcal{O}\times\mathbb{U}\times\mathbb{V}$, due to that $(\Xi,y,u,v)\in \widetilde{\mathcal{O}}$ and $\widetilde{\mathcal{O}}$ is a compact set in $[0,T]\times\mathbb{R}^n\times\mathbb{R}\times\mathbb{U}\times\mathbb{V}$, by Lemma \ref{xishuyizhi},
 \begin{align*}
  &\quad \sup\limits_{(\Xi,y,p,A)\in \mathcal{O}}|H_m^--H^-|\\
  &\le \sup\limits_{(\Xi,y,p,A)\in \mathcal{O}}\sup\limits_{u\in\mathbb{U},v\in\mathbb{V}}|g_m-g|\\
  &\le\sup\limits_{(\Xi,y,u,v)\in \widetilde{\mathcal{O}}}|g_m-g|\rightarrow 0,\quad m\rightarrow\infty.
 \end{align*}
The proof is completed.
\end{proof}
 Now bring up  stability nature of viscosity solution (such as Lemma 6.2 of \cite{Fleming2006}), which provides  foundation to connect $W$ to  HJBI equation \eqref{hjbi}.
\begin{prop}
\label{wending}Let $W_m$ be a viscosity subsolution (resp., supersolution) to the following PDE
$$\partial_tW_m+H_m^-(\Xi,W_m,DW_m,D^2W_m)=0,\quad (\Xi)\in[0,T]\times\mathbb{R}^n,$$
where $H_m^-(\Xi,y,p,A):\mathbb{H}\rightarrow\mathbb{R}$ is a continuous function satisfying the ellipticity setting
\begin{equation}\label{tuoyuan}
  H_m^-(\Xi,y,p,X)\le H_m^-(\Xi,y,p,Y), \quad X\le Y.
\end{equation}
Suppose $H_m^-\rightarrow H^-$, $W_m\rightarrow W$, $m\rightarrow\infty$,  uniformly in any compact set of their  domains. Then $W$ is a viscosity subsolution (resp., supersolution) of the PDE
$$\partial_tW+H^-(\Xi,W,DW,D^2W)=0,\quad (\Xi)\in[0,T]\times\mathbb{R}^n.$$
\end{prop}
We now begin to demonstrate the main results of this research. First we illustrate the existence of a viscosity solution to HJBI equation \eqref{hjbi}.
\begin{theorem}\label{cunzai}
  Suppose Hypotheses \ref{sdexishu} and \ref{bsdexishu} apply. Then $W$ is a viscosity solution to HJBI equation \eqref{hjbi}.
\end{theorem}
\begin{proof}
{\bf Step 1.} Suppose that $|g(\cdot,0,\cdot,\cdot)|$ is uniformly bounded on $[0,T]\times\mathbb{R}^n\times\mathbb{U}\times\mathbb{V}$. Then we have the Lipschitzian property of $g_m$. Indeed, for arbitrary $(\Xi,u,v)\in[0,T]\times\mathbb{R}^n\times\mathbb{U}\times\mathbb{V}$, $y,\widetilde{y}\in\mathbb{R}$,
\begin{align*}
  |g_m(\cdot,y,\cdot,\cdot)-g_m(\cdot,\widetilde{y},\cdot,\cdot)|&=|\int_{[-\frac{1}{m},\frac {1}{m}]}g(\cdot,y',\cdot,\cdot)(\zeta_m(y-y')-\zeta_m(\widetilde{y}-y'))dy'|\\
  &\le\max\limits_{y'\in[-\frac{1}{m},\frac{1}{m}]}|g(\cdot,y',\cdot,\cdot)|\int_{[-\frac{1}{m},\frac{1}{m}]}|\zeta_m(y-y')-\zeta_m(\widetilde{y}-y')|dy'\\
  &\overset{(c)}\le\max\limits_{y'\in[-\frac{1}{m},\frac{1}{m}]}(g|(\cdot,0,\cdot,\cdot)|+\eta(1+|y'|^p))\int_{[-\frac{1}{m},\frac{1}{m}]}C|y-\widetilde{y}|dy'\\
  &\le C|y-\widetilde{y}|,
\end{align*}
where (c) is by (iv) in Hypothesis \ref{bsdexishu} and the Lipschitzian property of $\zeta_m$. Then thanks to Theorem 4.2 in \cite{Buckdahn2008}, $W_m$ is a viscosity solution to the HJBI equation
\begin{equation}
  \left\{\begin{array}{ll}
  \partial _{t}W_m+H^{-}_m(\Xi,W_m,DW_m,D^{2}W_m)=0,\quad    (\Xi)\in [0,T)\times \mathbb{R}^n,  \\\\
   W_m|_{t=T}=h(x),\quad x\in \mathbb{R}^n,
\end{array} \right.
\end{equation}
where $W_m$ is defined in \eqref{moguangzhi} and $H^-_m$ in \eqref{moguangha}.

By virtue of Lemmas \ref{xishuyizhi}-\ref{hayizhi}, with $W_m$ fulling  ellipticity condition \eqref{tuoyuan}, we get by Proposition \ref{wending}, that $W$ is a viscosity solution to  limit equation \eqref{hjbi}.

{\bf Step 2.} $|g(\cdot,0,\cdot,\cdot)|$ is not necessarily uniformly bounded on $[0,T]\times\mathbb{R}^n\times\mathbb{U}\times\mathbb{V}$.

Set up a sequence of  cutoff functions
$$\overline{g}_k(\cdot,y,\cdot,\cdot):=g(\cdot,y,\cdot,\cdot)-g(\cdot,0,\cdot,\cdot)+\Lambda_k(g(\cdot,0,\cdot,\cdot)),\quad k\in\mathbb{N},$$
where $\Lambda_k(x)=\frac{\inf(k,|x|)}{|x|}x$.  Then it is straightforward to check that $\overline{g}_k$ satisfies Hypothesis \ref{bsdexishu}. Hence we can construct a series of
BSDEs
\begin{equation}\label{jieduanbsde}
 \overline{Y}_{s}^{\Xi,k;u,v}=h(x_{T}^{\Xi;u,v})+\int_{s}^{T}\overline{g}_k(r,x_{r}^{\Xi;u,v},\overline{Y}_{r}^{\Xi,k;u,v},u_r,v_r)dr-\int_{s}^{T}\overline{Z}_r^{\Xi,k;u,v}dB_{r},\quad t\le s\le T,
\end{equation}
with the cost function $$\overline{J}_k(\Xi;u,v):=\overline{Y}_{t}^{\Xi,k;u,v},$$
  value function $$\overline{W}_k:=\mathop{\essinf }\limits_{\beta \in {{\mathcal{B}}_{t,T}}}\mathop{\esssup }\limits_{u\in {{\mathcal{U}}_{t,T}}}\overline{J}_k(\Xi;u,\beta (u)),$$
 and  Hamiltonian
 \begin{align}&\overline{H}_{k}^{-}(\Xi,y,p,A)\nonumber\\
 =\:\;&\underset{u\in \mathbb{U}}{\mathop{\sup }}\,\underset{v\in \mathbb{V}}{\mathop{\inf }}\,\{\frac{1}{2}\Tr(\sigma {{\sigma }^{T}}(\Xi,u,v)A)+\langle p, b(\Xi,u,v)\rangle+\overline{g}_k(\Xi,y,u,v)\},\nonumber
\end{align}
for $(\Xi,y,p,A)\in\mathbb{H}.$
Since $\overline{g}_k(\cdot,0,\cdot,\cdot)=\Lambda_k(g(\cdot,0,\cdot,\cdot))$, $\overline{g}_k(\cdot,0,\cdot,\cdot)$ is uniformly bounded. Thus $\overline{g}_k$ satisfies the condition in Step 1, by which $\overline{W}_k$ is a viscosity solution to the HJBI equation
\begin{equation}
  \left\{\begin{array}{ll}
  \partial _{t}\overline{W}_k+\overline{H}^{-}_k(\Xi,\overline{W}_k,D\overline{W}_k,D^{2}\overline{W}_k)=0,\quad    (\Xi)\in [0,T)\times \mathbb{R}^n,  \\\\
   \overline{W}_k|_{t=T}=h(x),\quad x\in \mathbb{R}^n.
\end{array} \right.
\end{equation}
Then we prove that $\overline{g}_k\rightarrow g$, $\overline{Y}_{t}^{\Xi,k;u,v}\rightarrow Y_{t}^{\Xi;u,v}$, $\overline{W}_k\rightarrow W$ and $\overline{H}^-_k\rightarrow H^-$, uniformly in any compact set of their  domains as $k\rightarrow\infty$, where only the proof for $\overline{g}_k\rightarrow g$ as $k\rightarrow\infty$  needs to be altered from Lemma \ref{xishuyizhi} by cause of the difference between $\overline{g}_k$ and $g$, whereas the rest can be dealt with alike as Lemmas \ref{costyizhi}-\ref{hayizhi}.

As a matter of fact, for any  $\mathcal{O}$ compact in $[0,T]\times\mathbb{R}^n\times\mathbb{R}\times\mathbb{U}\times\mathbb{V}$ and for any $(\Xi,y,u,v)\in \mathcal{O}$, $g$ is bounded by a positive integer $M_\mathcal{O}$ in light of the continuity of $g$. Thus when $k\ge M_\mathcal{O}$,
$$\sup\limits_{(\Xi,y,u,v)\in \mathcal{O}}|\overline{g}_k(\Xi,y,u,v)-g(\Xi,y,u,v)|=\sup\limits_{(\Xi,y,u,v)\in \mathcal{O}}|g(\Xi,0,u,v)-\Lambda_k(g(\Xi,0,u,v))|=0,$$
which indicates that $\overline{g}_k\rightarrow g$ uniformly in every compact set of their domain as $k\rightarrow\infty$.

Employing one more time Proposition \ref{wending}, we get our results.\end{proof}
The unique weak solvability of   \eqref{hjbi} in  viscosity sense is discussed in
\begin{align*}
 	\Gamma &:= \big\{ W\in C([0,T]\times\mathbb{R}^n)|\text{ there  exists } \Upsilon>0, \text{ s.t. }\\
 	&\phantom{=\;\;}\lim_{|x|\rightarrow\infty}We^{-\Upsilon[\log(|x|^2+1)^{\frac{1}{2}}]^2}=0, \text{ uniformly in } t\in[0,T] \big\},
\end{align*}
initiated by Barles et al. \cite{Barles1997}. We begin  utilizing  comparison theorems of bounded viscosity solutions and then generalizing them to unbounded cases by a reference function to prove our results.

First we give the comparison  theorem of bounded viscosity solutions.
\begin{lemma}\label{bijiao}
Suppose Hypotheses \ref{sdexishu} and \ref{bsdexishu} apply. $W_1$ bounded from above is a viscosity subsolution to HJBI equation \eqref{hjbi} along with $W_2$   a viscosity supersolution to \eqref{hjbi} bounded from below. Then  $W_1\le W_2$ on  $(0,T]\times\mathbb{R}^n$.
\end{lemma}
 Analogous arguments to Proposition 5.1 in \cite{Zhuo2020} get our result, thus are omitted. Note that in \cite{Zhuo2020}, $\theta$ was generally chosen to be non-positive, otherwise the verification was simply  transformed to the former circumstance.

Subsequently introduce  reference function results of Lemma 5.1 in \cite{Buckdahn2008}.
\begin{lemma}\label{zhongyao}
  For arbitrarily fixed $\Upsilon>0$, there exists $\lambda>0$, s.t.
  $$\nu(\Xi):=e^{[\lambda(T-t)+\Upsilon]\kappa(x)}, \text{ with } \kappa(x)=\{\frac{1}{2}\log(|x|^2+1)+1\}^2$$
  satisfies
  $$\partial_t\nu+\sup\limits_{u\in\mathbb{U},v\in\mathbb{V}}\{\frac{1}{2}\Tr(\sigma {{\sigma }^{T}}(\Xi,u,v)D^2\nu)+\langle D\nu, b(\Xi,u,v)\rangle\}<0, \text{ in } [t_1,T]\times\mathbb{R}^n,$$
  where $t_1=T-\frac{\Upsilon}{\lambda}$.
\end{lemma}

Now we perform the uniqueness of the viscosity solution to HJBI equation \eqref{hjbi}.
\begin{theorem}\label{weiyi}
 Suppose Hypotheses \ref{sdexishu} and \ref{bsdexishu} apply. Then $W$ is the unique viscosity solution to HJBI equation \eqref{hjbi} in $\Gamma$.
\end{theorem}
\begin{proof}
  Let $W_1,W_2\in\Gamma$ be two viscosity solutions of HJBI equation \eqref{hjbi}. For arbitrary $\Upsilon>0$ and  $\varepsilon>0$, $W_1-\varepsilon\nu$ and $W_2+\varepsilon\nu$ are respectively viscosity subsolution bounded from above and super solution bounded from below on $[t_1,T]\times\mathbb{R}^n$. The boundedness property is easy to check. To see the viscosity property, fix $(\Xi)\in[t_1,T]\times\mathbb{R}^n$ and $\phi\in C^3_{b}([0,T]\times\mathbb{R}^n)$, s.t. $W_1-\phi$ attains its local maximum at $(\Xi)$, then we could also say that $W_1-\varepsilon\nu-(\phi-\varepsilon\nu)$ attains its local maximum at $(\Xi)$. Lemma \ref{zhongyao} leads to that
   \begin{equation}\label{zhongyaojieguo}
     \partial_t\nu+\frac{1}{2}\Tr(\sigma {{\sigma }^{T}}(\Xi,u,v)D^2\nu)+\langle D \nu, b(\Xi,u,v)\rangle<0 \text{ at } (\Xi), \text{ for any } u\in\mathbb{U}, v\in\mathbb{V}.
   \end{equation}
   Then
  \begin{align*}
    &\quad\partial_t(\phi-\varepsilon\nu)+H^-(\Xi,\phi-\varepsilon\nu,D(\phi-\varepsilon\nu),D^2(\phi-\varepsilon\nu))\\
    &=\sup\limits_{u\in\mathbb{U}}\inf\limits_{v\in\mathbb{V}}\big\{\partial_t(\phi-\varepsilon\nu)+\frac{1}{2}\Tr(\sigma\sigma^T(\Xi,u,v)D^2(\phi-\varepsilon\nu)) +\langle D(\phi-\varepsilon\nu),b(\Xi,u,v)\rangle\\&\qquad\qquad \quad +g(\Xi,\phi-\varepsilon\nu,u,v) \big\}\\
    &=\sup\limits_{u\in\mathbb{U}}\inf\limits_{v\in\mathbb{V}}\big\{\partial_t\phi+\frac{1}{2}\Tr(\sigma\sigma^T(\Xi,u,v)D^2\phi) +\langle D\phi, b(\Xi,u,v)\rangle+g(\Xi,\phi-\varepsilon\nu,u,v)\\&\qquad\qquad \quad -\varepsilon[\partial_t\nu+\frac{1}{2}\Tr(\sigma\sigma^T(\Xi,u,v)D^2\nu) +\langle D\nu, b(\Xi,u,v)\rangle]\big\}\\
    &\overset{(d)}{>}\partial_t\phi+H^-(\Xi,\phi,D\phi,D^2\phi)\\
    &>0,
    \end{align*}
    where $(d)$ is a result of \eqref{zhongyaojieguo}, the monotonicity of $g$ regarding $y$ and the fact that $\nu>0$. Hence $W_1-\varepsilon\nu$ is a viscosity subsolution to \eqref{hjbi}. The viscosity property of $W_2+\varepsilon\nu$ can be proved similarly. Then as a result of Lemma \ref{bijiao}, for any $(\Xi)\in(t_1,T]\times\mathbb{R}^n$,
$$W_1-\varepsilon\nu\le W_2+\varepsilon\nu.$$
Letting $\varepsilon\rightarrow 0$,  we get $$W_1\le W_2 \text{ on } (\Xi)\in(t_1,T]\times\mathbb{R}^n.$$ Similarly, $$W_1\ge W_2 \text{ on } (\Xi)\in(t_1,T]\times\mathbb{R}^n.$$
Thus $$W_1= W_2 \text{ on } (\Xi)\in(t_1,T]\times\mathbb{R}^n.$$
Letting $t_2=(t-\frac{\Upsilon}{\lambda})^+$,  we execute the same procedures on $[t_2,t_1]$ and then, if $t_2>0$ on $[t_3,t_2]$ where $t_3=(t_2-\frac{\Upsilon}{\lambda})^+$... etc. It eventually arrives at
$$W_1= W_2 \text{ on } (\Xi)\in[0,T]\times\mathbb{R}^n.$$
The proof is completed.
\end{proof}
\begin{remark}\label{daxiao}
  Notice that $H^-\le H^+$,   any viscosity solution to HJBI equation \eqref{upperhjbi} is a supersolution to HJBI equation \eqref{hjbi}. Consequently by  analogues arguments of  Theorem 3.4, $W\le U$, which is why we call $W$ a lower value  and $U$ an upper one. Moreover, if the Isaacs circumstance $H^-=H^+$ applies, then \eqref{upperhjbi} and \eqref{hjbi} coincide and  $W=U$, which is the value of the SDG.
\end{remark}
\section{Improved regularity of the value function}\label{sec4}

We intend to further improve the regularity of $W$ regarding $t$ to Lipschitzian continuity, which makes $W$ more likely to be differentiable regarding $t$.  The requirement in \cite{Buckdahn2011b} of uniformly bounded coefficients is here removed and  a more delicate estimate of $W$ is obtained. The methods in \cite{Buckdahn2011b} are invalid, due to the existence of nonanticipative strategies. Fortunately, via a novel nonanticipative strategy and some delicate analysis, we are able to overcome this difficulty.
\begin{prop}\label{yizhilip}
 For $b$ and $\sigma$, suppose they are continuous regarding $(u,v)$ uniformly in $(\Xi)$, and Lipschitzian regarding $(\Xi)$ uniformly in $(u,v)$. For $h$, suppose it is Lipschitzian regarding $x$. For $g$, suppose (I) or (II) applies, where (I): $g$ is independent of $z$, and (i), (iii) as well as (iv) in  Hypothesis \ref{bsdexishu} apply with additionally that $g$ is  Lipschitzian regarding $t$ uniformly in $(x,y,u,v)$; and (II): $g$ is continuous regarding $(u,v)$ uniformly in $(\Xi,y,z)$, and Lipschitzian regarding $(\Xi,y,z)$ uniformly in $(u,v)$.
Then for $\delta\in(0,T)$, there exist $C_\delta, C'>0$ s.t. lower value function $W$ satisfies the following property:
  $$|W(t_0,x_0)-W(t_1,x_1)|\le C_\delta(C'|t_0-t_1|+|x_0-x_1|), (t_0,x_0),(t_1,x_1)\in[0,T-\delta]\times\mathbb{R}^n,$$
  where $C'$ depends on $C$, $|x_0|$, $T$, as well as the upper bound of $|b(\cdot,0,\cdot,\cdot)|$, $|\sigma(\cdot,0,\cdot,\cdot)|$, $|g(\cdot,0,0,\cdot, \cdot)|$ (or $|g(\cdot,0,0,0,\cdot,\cdot)|$) on $[0,T]\times\mathbb{U}\times\mathbb{V}$, and $ C_\delta$  is independent of $(t_0,x_0)$ and $(t_1,x_1)$. Moreover, when the coefficients are uniformly bounded, $C'$ can only depend on  $T$.
\end{prop}
\begin{proof}
The proofs under conditions (I) are similar to (II), where in (I)  monotonicity condition  is applied instead of  Lipschitzian condition via It\^{o}'s formula. For this we only present the proof under condition (II).

Define $B^0_{s}=B_s-B_{t}$, $s\ge t$. Then $B^0$ is an $\mathcal{F}^0_s$-Brownian motion with $\mathcal{F}^0_s=\sigma\{B^0_r,t\le r\le s\}\vee\mathcal{N}$. Denote $\mathcal{U}^t$ as the set of admissible controls in $\mathcal{U}_{t,T}$ that are adapted to $\mathcal{F}^0_s$, $s\in[t,T]$. Similarly define $\mathcal{V}^t$. And denote $\mathcal{B}^t$ as the set of nonanticipative strategies from $\mathcal{U}^t$ to $\mathcal{V}^t$. Then for $u\in\mathcal{U}^t$ and $v\in\mathcal{V}^t$, $Y^{\Xi;u,v}_t$ is deterministic. Define a lower value function
$$\mathcal{W}(\Xi)=\underset{\beta\in\mathcal{B}^t}{\inf}\underset{u\in\mathcal{U}^t}{\sup}Y^{\Xi;u,\beta(u)}_t.$$
Then  the   proofs of Theorems 3.6 and  4.2 in \cite{Buckdahn2008} still hold to obtain that $\mathcal{W}(\Xi)$ is a viscosity solution of HJBI equation \eqref{hjbiz}. Therefore, by the uniqueness of viscosity solution, we obtain that $W(\Xi)=\underset{\beta\in\mathcal{B}^t}{\inf}\underset{u\in\mathcal{U}^t}{\sup}Y^{\Xi;u,\beta(u)}_t$. Hence, we only need to consider the elements in $\mathcal{U}^t$ and $\in\mathcal{B}^t$ to continue our analysis.

{\bf Step 1.}  Through time change, we make  state-cost systems SDE \eqref{sdg} and BSDE \eqref{controlled} that are individually starting at $t_0$ and $t_1$,  uniformly start at $t_0$, hence comparable.

For $t_1<T$ and $t=t_0< T$, set the time change
$\tau:[t_1,T]\ni s\rightarrow t_0+\frac{T-t_0}{T-t_1}(s-t_1)\in [t_0,T]$. Then $\partial_s\tau=\frac{T-t_0}{T-t_1}:=\dot{\tau}$. Introduce $\mathbb{B}_s:=\frac{B^0_{\tau(s)}}{\sqrt{\dot{\tau}}}$, $s\in[t_1,T]$, which is a Brownian motion beginning at $t_1$. For $u^0\in\mathcal{U}^{t_0}$ and $v^0\in\mathcal{V}^{t_0}$, define $u^1_r:=u^0_{\tau(r)}$ and $v^1_r:=v^0_{\tau(r)}$, $r\in[t_1,T]$. Then it is obvious that $u^1\in\mathcal{U}_{\mathbb{B}}^{t_1}$ and $v^1\in\mathcal{V}_{\mathbb{B}}^{t_1}$, where $\mathcal{U}_{\mathbb{B}}^{t_1}$ (or $\mathcal{V}_{\mathbb{B}}^{t_1}$) is the set of $\mathbb{U}$ (or $\mathbb{V}$)-valued processes adapted to $\mathcal{F}^\mathbb{B}_s:=\sigma\{\mathbb{B}_r,t_1\le r\le s\}\vee\mathcal{N}$. Then for $x_0, x_1\in\mathbb{R}^n$,   introduce the following equations, which are uniquely solvable due to the hypotheses:
\begin{equation*}
  x^0_s=x_0+\int_{t_0}^sb(r,x^0_r,u^0_r,v^0_r)dr+\int_{t_0}^s\sigma(r,x^0_r,u^0_r,v^0_r)dB^0_r,
  \end{equation*}
\begin{equation*}
  x^1_s=x_1+\int_{t_1}^sb(r,x^1_r,u^1_r,v^1_r)dr+\int_{t_1}^s\sigma(r,x^1_r,u^1_r,v^1_r)d\mathbb{B}_r,
  \end{equation*}
  $$Y^0_s=h(x^0_T)+\int_s^Tg\left(r,x^0_r,Y^0_r,Z^0_r,u^0_r,v^0_r\right)dr-\int_s^TZ^0_rd B^0_r,$$
$$Y^1_s=h(x^1_T)+\int_s^Tg\left(r,x^1_r,Y^1_r,Z^1_r,u^1_r,v^1_r\right)dr-\int_s^TZ^1_rd\mathbb{B}_r.$$
Set $\tilde{x}^1_s=x^1_{\tau^{-1}(s)}$, where $\tau^{-1}$ is the inverse function of $\tau$. Due to that $\mathbb{B}_{\tau^{-1}(s)}=\frac{1}{\sqrt{\dot{\tau}}}B^0_s$, and $u^1_{\tau^{-1}(s)}=u^0_s$, $v^1_{\tau^{-1}(s)}=v^0_s$, $s\in[t_0,T]$, we obtain that $\tilde{x}^1$ satisfies the following SDE:
\begin{equation*}
  \tilde{x}^1_s=x_1+\int_{t_0}^s\frac{1}{\dot{\tau}}b(\tau^{-1}(r),\tilde{x}^1_r,u^0_r,v^0_r)dr+\int_{t_0}^{s}\frac{1}{\sqrt{\dot{\tau}}}\sigma(\tau^{-1}(r),\tilde{x}^1_r,u^0_r,v^0_r)dB^0_r.
\end{equation*}
Set $\tilde{Y}^1_s=Y^1_{\tau^{-1}(s)}$ and $\tilde{Z}^1_s=\frac{1}{\sqrt{\dot{\tau}}}Z^1_{\tau^{-1}(s)}$, then we obtain that
$$\tilde{Y}^1_s=h(\tilde{x}^1_T)+\int_s^T\frac{1}{\dot{\tau}}g\left(\tau^{-1}(r),\tilde{x}^1_r,\tilde{Y}^1_r,\sqrt{\dot{\tau}}\tilde{Z}^1_r,u^0_r,v^0_r\right)dr-\int_s^T\tilde{Z}^1_rdB^0_r.$$

For  $s\in[t_0,T]$, It\^{o}'s formula  results in
\begin{align}
  & \nonumber\quad E\left[\underset{r\in[t_0,s]}{\sup}\left|x^0_r-\tilde{x}^1_r\right|^2\right]\\
  &\label{ineq}\le 4E\bigg[|x_0-x_1|^2+\int_{t_0}^s|b(r,x^0_r,u^0_r,v^0_r)-\frac{1}{\dot{\tau}}b(\tau^{-1}(r),\tilde{x}^1_r,u^0_r,v^0_r)|^2dr\\
  &\nonumber+4\int_{t_0}^s|\sigma(r,x^0_r,u^0_r,v^0_r)-\frac{1}{\sqrt{\dot{\tau}}}\sigma(\tau^{-1}(r),\tilde{x}^1_r,u^0_r,v^0_r)|^2dr\bigg].
\end{align}For the second integral on the right side of \eqref{ineq}, \begin{align}
&\ |b(r,x^0_r,u^0_r,v^0_r)-\frac{1}{\dot{\tau}}b(\tau^{-1}(r),\tilde{x}^1_r,u^0_r,v^0_r)|\nonumber\\
\le&\ |b(r,x^0_r,u^0_r,v^0_r)-\frac{1}{\dot{\tau}}b(r,x^0_r,u^0_r,v^0_r)|+|\frac{1}{\dot{\tau}}b(r,x^0_r,u^0_r,v^0_r)-\frac{1}{\dot{\tau}}b(\frac{1}{\tau},\tilde{x}^1_r,u^0_r,v^0_r)|\nonumber\\
\le&\  C(1+|x^0_r|)|1-\frac{1}{\dot{\tau}}|+\frac{1}{\dot{\tau}}C(|r-\frac{1}{\tau}|+|x^0_r-\tilde{x}^1_r|)\nonumber\\
\le&\  C(1+|x^0_r|)|1-\frac{1}{\dot{\tau}}|+\frac{1}{\dot{\tau}}C(|r-\frac{1}{\tau}|+\underset{s\in[t_0,r]}{\sup}|x^0_s-\tilde{x}^1_s|).\label{bsigma}
\end{align}
The third integral on the right side of \eqref{ineq} is similarly estimated with $b$ substituted by $\sigma$.
 Lemma 2.1 in \cite{Buckdahn2011b} leads to that there exists a $C_{\delta}'>0$ only depending on $\delta$ and $T$, s.t. \begin{align}\label{tau}\left|1-\frac{1}{\dot{\tau}}\right|+|\tau^{-1}(r)-r|+|1-\sqrt{\dot{\tau}}|\le C_{\delta}'|t_0-t_1|, \ r\in[t_0,T].\end{align}
Plugging \eqref{tau} and \eqref{bsigma} into \eqref{ineq}, and due to Gronwall's inequality as well as estimates of $|x^0
_r|$ (e.g. Proposition 2.2 in \cite{Pu2018}), we arrive at
\begin{align}\label{deltax} E\left[\underset{r\in[t_0,s]}{\sup}\left|x^0_r-\tilde{x}^1_r\right|^2\right]
\le C_\delta''(|x_0-x_1|^2+C''|t_0-t_1|^2),\end{align} where $C''$ depends on $C$, $|x_0|$, $T$, as well as the upper bound of $|b(\cdot,0,\cdot,\cdot)|$, $|\sigma(\cdot,0,\cdot,\cdot)|$ on $[0,T]\times\mathbb{U}\times\mathbb{V}$, and $C_\delta''$ is independent of $(t_0,x_0)$, $(t_1,x_1)$, $u^0$ and $v^0$.

Subsequently,  It\^{o}'s formula  leads to \begin{align*}
 &  \quad E\left[\left|Y^0_{t_0}-\tilde{Y}_{t_0}^1\right|^2+\int_{t_0}^T\left|Z^0_r-\tilde{Z}_r^1\right|^2dr-\left|h(x^0_T)-h(\tilde{x}^1_T)\right|^2\right]\\
 &=E\left[2\int_{t_0}^T(Y^0_r-\tilde{Y}^1_r)\left[g(r,x^0_r,Y^0_r,Z^0_r,u^0_r,v^0_r)-\frac{1}{\dot{\tau}}g\left(\tau^{-1}(r),\tilde{x}^1_r,\tilde{Y}^1_r,\sqrt{\dot{\tau}}\tilde{Z}^1_r,u^0_r,v^0_r\right)\right]dr\right]\\
 &\le E\left[2\int_{t_0}^T\left|\left(1-\frac{1}{\dot{\tau}}\right)\left(Y^0_r-\tilde{Y}^1_r\right)g(r,x^0_r,Y^0_r,Z^0_r,u^0_r,v^0_r)\right|dr\right]\\
 &+E\left[2\int_{t_0}^T\frac{C}{\dot{\tau}}\left|Y^0_r-\tilde{Y}^1_r\right|\left(\left|r-\tau^{-1}(r)\right|+\left|x^0_r-\tilde{x}^1_r\right|+\left|Y^0_r-\tilde{Y}^1_r\right|+\left|Z^0_r-\sqrt{\dot{\tau}}\tilde{Z}^1_r\right|\right)dr\right].
\end{align*}Since $\sqrt{\dot{\tau}}+\frac{1}{\dot{\tau}}\le C_T$, where $C_T>0$ only depends on $T$, by some adjustments as well as estimates of $(Y^0, Z^0)$ (Proposition \ref{yilai}) and \eqref{deltax},
we arrive at $|Y^0_{t_0}-\tilde{Y}^1_{t_0}|\le C_\delta(C'|t_0-t_1|+|x_0-x_1|)$,  where $C'$ depends on $C$, $|x_0|$, $T$, as well as the upper bound of $|b(\cdot,0,\cdot,\cdot)|$, $|\sigma(\cdot,0,\cdot,\cdot)|$, $|g(\cdot,0,0,0,\cdot,\cdot)|$ on $[0,T]\times\mathbb{U}\times\mathbb{V}$, and $C_\delta$ is independent of $(t_0,x_0)$, $(t_1,x_1)$, $u^0$ and $v^0$.

{\bf Step 2.} We begin the proof of the main result.

First, define two value functions
\begin{equation*}
  W(t_0,x_0):=\underset{\beta^0\in\mathcal{B}^{t_0}}{\inf}\sup\limits_{u^0\in\mathcal{U}^{t_0}}Y^0_{t_0}[u^0,\beta^0(u^0)],\ W(t_1,x_1):=\underset{\beta^1\in\mathcal{B}^{t_1}}{\inf}\sup\limits_{u^1\in\mathcal{U}_{\mathbb{B}}^{t_1}}Y^1_{t_1}[u^1,\beta^1(u^1)],
\end{equation*}
where $Y^0_{t_0}[u^0,\beta^0(u^0)]$ emphasizes the dependence on $(u^0,\beta^0(u^0))$.  $Y^1_{t_1}[u^1,\beta^1(u^1)]$ is similar.\\
Next we prove that $\underset{\beta^1\in\mathcal{B}^{t_1}}{\inf}\sup\limits_{u^1\in\mathcal{U}_{\mathbb{B}}^{t_1}}Y^1_{t_1}[u^1,\beta^1(u^1)]=\underset{\beta^0\in\mathcal{B}^{t_0}}{\inf}\sup\limits_{u^0\in\mathcal{U}^{t_0}}\tilde{Y}^1_{t_0}[u^0,\beta^0(u^0)]$.

 For $\beta^1\in\mathcal{B}^{t_1}$, define a mapping $\tilde{\beta}^1$ on $\mathcal{U}^{t_0}$ s.t., $\tilde{\beta}^1(u^0)(\cdot)=\beta^1(u^0_\tau)(\tau^{-1}(\cdot))$, $\forall u^0\in\mathcal{U}^{t_0}$. Then we can obtain that $\tilde{\beta}^1$ maps $\mathcal{U}^{t_0}$ into $\mathcal{V}^{t_0}$. Moreover, $\tilde{\beta}^1$ can be verified to be nonanticipative.

Indeed, we first prove that for an $\mathcal{F}^0$-stopping time $\tilde{s}$, $\tau^{-1}(\tilde{s})$ is an $\mathcal{F}^{\mathbb{B}}$-stopping time.

We know that $\mathbb{B}_s=\frac{B^0_{\tau(s)}}{\sqrt{\dot{\tau}}}$, $s\in[t_1,T]$. Letting $\tau(s)=r$, then $\mathbb{B}_{\tau^{-1}(r)}=\frac{B^0_{r}}{\sqrt{\dot{\tau}}}$, $r\in[t_0,T]$. Since $\dot{\tau}$ is a constant, we obtain that $\mathcal{F}^{\mathbb{B}}_{\tau^{-1}(r)}=\mathcal{F}^0_r$. Therefore, for any $s\in[t_1,T]$, \begin{align*}
\{\tau^{-1}(\tilde{s})\le s\}=
\{\tau^{-1}(\tilde{s})\le\tau^{-1}(r)\}=\{\tilde{s}\le r\}\in\mathcal{F}^0_r=\mathcal{F}^{\mathbb{B}}_{\tau^{-1}(r)}=\mathcal{F}^{\mathbb{B}}_{s},
\end{align*}
which means that $\tau^{-1}(\tilde{s})$ is an $\mathcal{F}^{\mathbb{B}}$-stopping time.

Therefore, for $u^0,\tilde{u}^0\in\mathcal{U}^{t_0}$, s.t., $u^0\equiv \tilde{u}^0$ on $[t_0,\tilde{s}]$,  it holds that $u^0_{\tau}\equiv \tilde{u}^0_{\tau}$ on $[t_1,\tau^{-1}(\tilde{s})]$. Thus $\beta^1(u^0_{\tau})\equiv\beta^1(\tilde{u}^0_{\tau})$ on $[t_1,\tau^{-1}(\tilde{s})]$, and
$\beta^1(u^0_{\tau})(\tau^{-1})\equiv\beta^1(\tilde{u}^0_{\tau})(\tau^{-1})$ on $[t_0,\tilde{s}]$, which proves that $\tilde{\beta}^1\in\mathcal{B}^{t_0}$.

Consequently,
\begin{align*}
&\quad \underset{\beta^1\in\mathcal{B}^{t_1}}{\inf}\sup\limits_{u^1\in\mathcal{U}_{\mathbb{B}}^{t_1}}Y^1_{t_1}[u^1,\beta^1(u^1)]\\
&=\underset{\beta^1\in\mathcal{B}^{t_1}}{\inf}\sup\limits_{u^0\in\mathcal{U}_{\mathbb{B}}^{t_0}}Y^1_{t_1}[u^0_\tau,\beta^1(u^0_\tau)]\\
&=\underset{\beta^1\in\mathcal{B}^{t_1}}{\inf}\sup\limits_{u^0\in\mathcal{U}_{\mathbb{B}}^{t_0}}\tilde{Y}^1_{t_0}[u^0,\beta^1(u^0_\tau)(\tau^{-1})]\\
&=\underset{\tilde{\beta}^1\in\mathcal{B}^{t_0}}{\inf}\sup\limits_{u^0\in\mathcal{U}_{\mathbb{B}}^{t_0}}\tilde{Y}^1_{t_0}[u^0,\tilde{\beta}^1(u^0)]\\
&=\underset{\beta^0\in\mathcal{B}^{t_0}}{\inf}\sup\limits_{u^0\in\mathcal{U}_{\mathbb{B}}^{t_0}}\tilde{Y}^1_{t_0}[u^0,\beta^0(u^0)],
\end{align*}
where we have used that $\{\tilde{\beta}^1|\beta^1\in\mathcal{B}^{t_1}\}=\mathcal{B}^{t_0}$, and  $\tilde{Y}^1_{t_0}=Y^1_{t_1}$.

Since
\begin{equation*}
  W(t_0,x_0)=\underset{\beta^0\in\mathcal{B}^{t_0}}{\inf}\sup\limits_{u^0\in\mathcal{U}^{t_0}}Y^0_{t_0}[u^0,\beta^0(u^0)],\ W(t_1,x_1)=\underset{\beta^0\in\mathcal{B}^{t_0}}{\inf}\sup\limits_{u^0\in\mathcal{U}_{\mathbb{B}}^{t_0}}\tilde{Y}^1_{t_0}[u^0,\beta^0(u^0)],
\end{equation*}
similar to the proof of Lemma \ref{zhiyizhi},
 we have for any $\varepsilon>0$, there exist some $\beta_1^{\varepsilon},\beta^\varepsilon\in\mathcal{B}^{t_0}$ and $u_1^\varepsilon,u^\varepsilon\in\mathcal{U}^{t_0}$, s.t.,
\begin{align*}
  &\quad |W(t_0,x_0)-W(t_1,x_1)|-4\varepsilon\nonumber\\
  &\le \left|Y^{0}_{t_0}[u^\varepsilon,\beta^{\varepsilon}_1(u^\varepsilon)]-\tilde{Y}^1_{t_0}[u^\varepsilon,\beta_1^\varepsilon(u^\varepsilon)]\right|+\left|Y^{0}_{t_0}[u^{\varepsilon}_1,\beta^\varepsilon(u^{\varepsilon}_1)]-\tilde{Y}^1_{t_0}[u_1^\varepsilon,\beta^\varepsilon(u_1^\varepsilon)]\right|\\
  &\le  C_\delta(C'|t_0-t_1|+|x_0-x_1|).
\end{align*}
  From the arbitrariness of $\varepsilon$, $|W(t_0,x_0)-W(t_1,x_1)|\le C_\delta(C'|t_0-t_1|+|x_0-x_1|)$.  In addition, in {\bf Step 1}, when the coefficients are uniformly bounded, it is easy to see that $C'$ can only depend on  $T$.

\end{proof}

\section{Example}\label{sec5}
Here we present three examples to illustrate the effectiveness of our theoretical results. The first  example gives an optimal control-strategy pair via  the classical solution of an HJBI equation, the second example demonstrates the analytical form of a viscosity solution of an HJBI equation, and the last example illustrates an application in financial portfolio optimization. Since concrete examples are fairly rare in the study of recursive control-versus-strategy SDGs, our examples are novel compared to existing literature.

\textbf{Example 5.1.} Consider a lower game. Suppose $\mathbb{U}=[0,1]$, $\mathbb{V}=[-1,0]$, $b(\cdot)=u+v$, $\sigma(\cdot)=v$, $g(\cdot)=-(u+v)x-u-v$, $h(\cdot)=\frac{1}{2}x^2$, where $x\in\mathbb{R}$. Then the corresponding HJBI equation is
\begin{equation}
  \left\{\begin{array}{ll}
  \partial _{t}W+\underset{u\in[0,1]}{\sup}\underset{v\in[-1,0]}{\inf}\left\{\frac{1}{2}v^2D^2W+(u+v)DW-(u+v)x-u-v\right\}=0,\quad    \\\\
   {W}|_{t=T}=\frac{1}{2}x^2.\quad
\end{array} \right.
\label{li61}\end{equation}
It can be verified that $W(\Xi)=\frac{1}{2}x^2$ is the classical solution of HJBI equation~\eqref{li61}. By verification theorem, we can obtain that $(\hat{u}(\cdot),\hat{\beta}(\cdot))=(0,0)$ is an optimal control-strategy pair. Notice that $W$ does not satisfy linear growth condition w.r.t. $x$. However, it can be easily verified that the comparison theorem for BSDE still holds here, which makes the verification theorem valid.

\textbf{Example 5.2.} Consider a lower game. Suppose $\mathbb{U}=[0,1]$, $\mathbb{V}=[-1,0]$, $b(\cdot)=x+xv+xuv$, $\sigma(\cdot)=vx$, $g(\cdot)=-(u+v)z-u$, $h(\cdot)=-|x|$, $x\in\mathbb{R}$. Then the corresponding HJBI equation is
\begin{equation}\label{li62}
  \left\{\begin{array}{ll}
  \partial _{t}W+xDW+\underset{u\in[0,1]}{\sup}\underset{v\in[-1,0]}{\inf}\left\{\frac{1}{2}x^2v^2D^2W+xvDW-xv^2DW-u\right\}=0,\quad    \\\\
   {W}|_{t=T}=-|x|.\quad
\end{array} \right.
\end{equation}
It can be verified that the viscosity solution of HJBI equation \eqref{li62} is
\begin{equation*}
W(\Xi)=\left\{\begin{array}{ll}
 -e^{T-t}x, \text{if~}  x>0,  \\\\
  e^{T-t}x, \;\;\;\text{if~} x\le 0.
\end{array} \right.
\end{equation*}
Indeed, when $x>0$ or $x<0$, $W$ is smooth, and therefore $W$ is the viscosity solution of HJBI equation \eqref{li62}. When $x=0$, $W$ is not differentiable w.r.t. $x$. For any $\phi\in C^3_{b}([0,T]\times\mathbb{R})$, s.t., $W-\phi$ attains its local maximum at $(t_0,0)$, where $t_0\in[0,T)$, it can be verified that $\partial_t\phi(t_0,0)=0$. Then by Definition \ref{nxjdingyi} and the uniqueness of viscosity solution,  $W$ is the unique viscosity solution of HJBI equation \eqref{li62}.

\textbf{Example 5.3.}
We give a financial application of an upper game.
  Suppose an investor's portfolio consists  of one stock and one liquid bond. The stock price $S_r(r\ge t)$ follows a geometric Brownian motion
  \begin{equation}\label{stock}
    \frac{dS_r}{S_r}=b_rdr+\sigma_rdB_r, \quad\text{ where } S_t>0 \text{ is given,}
  \end{equation}
  along with the bond price $s_r(r\ge t)$ satisfying an ordinary differential equation (ODE)
  \begin{equation}\label{bond}
    ds_t=r_ts_tdt, \quad s_t=1.
  \end{equation}
  In SDE \eqref{stock} and ODE \eqref{bond}, $r$, the  bond's interest rate, $b$, the stock's appreciation rate and $\sigma$, the volatility rate, are all continuous deterministic functions.  Now suppose that the investor manages to maximize her utility preference by  controlling the  investment and consumption. Denote the portion of  wealth invested in stock by  $u^1:\Omega\times[t,T]\rightarrow[-1,1]$, the portion of  wealth in  bond  $u^0:=1-u^1$, and the money spent on consumption  $c:\Omega\times[t,T]\rightarrow[M,N]$, where $N> M\ge0$.  The market control manages to neutralize the investor's gain by friction such as transaction costs and competition, etc. Thus to characterize the market friction we diminish each wealth allocation $u^i$ by a factor $v^i\in[0,1]$. Assume that all the control processes are $\{\mathcal{F}_r\}_{t\le r\le T}$-adapted.

  Then the net wealth satisfies
  \begin{equation}\label{wealth}
 dV_t=(v^0_tu^0_tr_tV_t+v^1_tu^1_tb_tV_t-c_t)dt+\sigma_tv^1_tu^1_tV_tdB_t, \quad V_t=x,
  \end{equation}
  $x$ being  the investor's original  asset, which is positive. Then obviously  \eqref{wealth} fulfills Hypothesis \ref{sdexishu}.

    Suppose  the  utility preference of the investor is an Epstein-Zin utility $Y_t$ satisfying the BSDE
    \begin{equation}\label{epstein}
      dY_t=-\frac{\rho}{1-\frac{1}{\varsigma}}(1-\vartheta)Y_t[(\frac{c_t}{((1-\vartheta)Y_t)^{\frac{1}{1-\vartheta}}})^{1-\frac{1}{\varsigma}}-1]dt+Z_tdB_t,\quad Y_T=h(V_T),    \end{equation}
 in which $h$ is  Lipschitzian as well as  deterministic, and $\rho,\vartheta$ along with $\varsigma$ are illustrated in the Introduction.  Notice that $Y_t$ is a $\mathcal{F}_t$-measurable r.v.

The generator in BSDE \eqref{epstein} in general does not fulfill the Lipschitzian condition regarding the utility preference $Y$ and the consumption $c$, by which our study is suitable in some circumstances.  \cite{Kraft2013} proposes four circumstances where the generator is monotonic regarding the utility $Y$ and two of them fulfill the polynomial growth condition (iv) in Hypothesis \ref{bsdexishu}: (i) $\vartheta>1$ and $\varsigma>1$; (ii) $\vartheta<1$ and $\varsigma<1$. Then it is not difficult to  choose proper power s.t. the generator of BSDE \eqref{epstein} is continuous and monotonic but non-Lipschitzian regarding the utility $Y$ in both circumstances. Regarding  the continuity regarding the consumption $c$, if $M>0$, both (i) and (ii) are Lipschitzian. If $M=0$, then only (i) fulfills the continuous but non-Lipschitzian condition regarding the consumption $c$.

Let $\mathcal{U}_{t,T}$ comprise feasible controls $u:=(u^1,u^0,c)$ for the investor and $\mathcal{V}_{t,T}$ for the market friction $v:=(v^0,v^1)$.  $\mathcal{A}_{t,T}:=\{\alpha:\mathcal{V}_{t,T}\rightarrow\mathcal{U}_{t,T}\}$ comprises nonanticipative strategies for the investor.
The purpose of the investor is to maximize her utility despite the market friction. The value function is \begin{equation}
  U:=\underset{\alpha\in\mathcal{A}_{t,T}}\esssup\,\underset{v\in\mathcal{V}_{t,T}}\essinf\,  Y_t=\sup\limits_{\alpha\in\mathcal{A}_{t,T}}\inf\limits_{v\in\mathcal{V}_{t,T}}  E[Y_t],
\end{equation}
where the last equality is due to the result of  Remark \ref{queding}  in the upper game case.
Consequently for all suitable non-Lipschitzian conditions satisfying Hypothesis \ref{bsdexishu}, we can apply the results of  Theorems \ref{cunzai} and \ref{weiyi}  in the upper game case, to verify that $U$ is the unique viscosity solution to the following HJBI equation
\begin{equation}\label{lizi}
\left\{\begin{array}{ll}
  \partial_t U+\inf\limits_{(v^0,v^1)\in[0,1]^2}\sup\limits_{(u^1,c)\in[-1,1]\times[M,N]}\big\{[(v^0(1-u^1)r_t+v^1u^1b_t)]xDU+\frac{1}{2}|u^1v^1\sigma_tx|^2D^2U
  \\ \qquad\qquad\qquad\qquad\qquad\qquad\qquad\qquad\qquad+\frac{\rho}{1-\frac{1}{\varsigma}}(1-\vartheta)U[(\frac{c}{((1-\vartheta)U)^{\frac{1}{1-\vartheta}}})^{1-\frac{1}{\varsigma}}-1] \big\}=0,\\
  U|_{t=T}=h(x).
\end{array}\right.
\end{equation}
The analysis above is summarized as below.
\begin{prop}
 HJBI equation \eqref{lizi} admits a unique viscosity solution with monotonic and continuous coefficients satisfying Hypothesis \ref{bsdexishu}, which are illustrated in Section \ref{sec4}.
\end{prop}
\section{Conclusion}
In this research, we study the DPP of two-player zero-sum recursive SDGs, based on which we prove the unique solvability of the derived HJBI equations in the viscosity sense employing mollification and truncation techniques  as well as stability property and comparison theorem of viscosity solution.
In future, we are interested in  working out the case in which  $g$ depends on $Z$, and  weak solutions of \eqref{hjbiz} and \eqref{upperhjbiz} in  Sobolev sense as \cite{Wei2014}. Moreover, the subsequent verification theorem via viscosity solution is also worth exploring for the characterization of optimal control-strategy pairs.
\section*{ Declaration of interest statement}The authors declare that they have no  competing financial interests or personal relationships to influence the study in this paper.
\section*{Acknowledgement}
The authors would like to express their appreciation to  two anonymous reviewers for their valuable comments and advices, which help improve this paper greatly.

\end{document}